\documentclass{amsart}

\usepackage{amsmath, amssymb, amsbsy}
\usepackage{amsthm}
\usepackage{thmtools}
\usepackage{array}
\usepackage{graphpap, color, paralist,xcolor}
\usepackage[mathscr]{eucal}
\usepackage[pdftex]{graphicx}
\usepackage[pdftex,colorlinks,backref=page,citecolor=blue]{hyperref}
\usepackage{pifont}
\usepackage{tikz, tcolorbox}
\tikzset{sdot/.style = {fill, circle, inner sep = 1.25pt}}
\usepackage{mathtools}
\usepackage{cleveref}
\usepackage{float}
\usepackage{crossreftools}

\newcounter{bullet}

\usepackage{setspace}
\setdisplayskipstretch{1}
\newenvironment{subproof}[1][\proofname]{
  
  \begin{proof}[#1]
}{
  \end{proof}
}
\usepackage{geometry}
\usepackage{comment}
\newtheorem{thm}{Theorem}[section]

\newtheorem*{prop*}{Proposition}
\newtheorem{cor}[thm]{Corollary}
\newtheorem{lem}[thm]{Lemma}
\newtheorem*{lem*}{Lemma}

\theoremstyle{definition}
\newtheorem{mydef}[thm]{Definition}
\newtheorem{claim}[thm]{Claim}

\newtheorem{remark}[thm]{Remark}
\newtheorem{obs}[thm]{Observation}

\newcommand{\PP}{\mathbb{P}}
\newcommand{\RR}{\mathbb{R}}
\newcommand{\ZZ}{\mathbb{Z}}
\newcommand{\EE}{\mathbb{E}}
\newcommand{\Z}{\mathbb{Z}}

\newcommand{\cO}{\mathcal{O} }
\newcommand{\cA}{\mathcal{A} }
\newcommand{\cB}{\mathcal{B} }

\newcommand{\cE}{\mathcal{E} }
\newcommand{\cF}{\mathcal{F} }
\newcommand{\cG}{\mathcal{G} }
\newcommand{\cH}{\mathcal{H} }

\newcommand{\cU}{\mathcal{U} }
\newcommand{\cV}{\mathcal{V} }
\newcommand{\cW}{\mathcal{W} }

\newcommand{\qd}{\mathrm{quad}}

\newcommand{\md}{\mathrm{mid}}

\newcommand{\beq}[1]{\begin{equation}\label{#1}}
\newcommand{\enq}[0]{\end{equation}}

\newcommand{\gd}[0]{\delta }

\newcommand{\sub}[0]{\subseteq}

\newcommand{\Hom}[0]{\mbox{\rm{Hom}}}

\newcommand{\dist}[0]{\mbox{\rm{dist}}}
\newcommand{\diam}[0]{\mbox{\rm{diam}}}

\newenvironment{proof*}[1][\proofname]{
  
  \begin{proof}[#1]}{\end{proof}}

\begin{document}

\title{Range of random $\mathbb Z$-homomorphisms on weak expanders}

\author[D. Dong]{Dingding Dong}
\address{Department of Mathematics, California Institute of Technology}
\email{ddong124@caltech.edu}

\author[J. Park]{Jinyoung Park}
\address{Department of Mathematics, Courant Institute of Mathematical Sciences, New York University}
\email{jinyoungpark@nyu.edu}

\begin{abstract}
We prove that random $\mathbb{Z}$-homomorphisms on weakly expanding bipartite graphs exhibit a strong ``flatness" phenomenon. Extending prior work of Peled, Samotij, and Yehudayoff for expanders, we first show that on any bipartite $(n, d, \lambda)$-graph with $\lambda \le (1-\delta)d$, a uniformly chosen $\mathbb{Z}$-homomorphism has a range at most $O(\log \log n)$ with high probability, which is tight up to a constant factor. This provides an affirmative answer to their question in the spectral setting. As a concrete application, we prove that a random $\mathbb{Z}$-homomorphism on the middle layers of the Hamming cube takes at most $5$ values with high probability. This shows that the $O(1)$-flatness for the full Hamming cube, proved by Kahn and Galvin, persists even when the rigid structural properties are relaxed.
\end{abstract}

\maketitle

\section{Introduction}

Discrete height functions, particularly integer-valued Lipschitz functions and $\mathbb Z$-homomor\-phisms on graphs, are widely studied at the intersection of combinatorics, probability, and statistical mechanics. The primary focus of this paper is the ``flatness'' phenomenon of random $\mathbb{Z}$-homomorphisms on weak expanders. Given a graph $\Gamma$, a \textit{$\mathbb{Z}$-homomorphism} is a function $h:V(\Gamma) \rightarrow \mathbb Z$ such that adjacent vertices receive adjacent integer values (i.e., $|h(u)-h(v)|=1$ for all edges $\{u,v\}$). Clearly, such a function can only exist if the underlying graph $\Gamma$ is bipartite. This framework provides a natural generalization of the simple random walk on $\Z$; indeed, an $n$-step simple random walk corresponds exactly to a uniformly random $\Z$-homomorphism from a path of length $n$ to $\Z$. Extending this probabilistic perspective, a uniformly random $\mathbb Z$-homomorphism from a general bipartite graph $\Gamma$, rooted by mapping a fixed vertex $v_0$ to $0$, is also called a \emph{$\Gamma$-indexed random walk} (see, e.g.,~\cite{benjamini2000random,benjamini1994tree}).

Formally, let $\Gamma=(V, E)$ be a connected bipartite graph with vertex bipartition $V=\cE\cup\cO$. For a distinguished vertex $v_0\in V$, we define
\[
\Hom_{v_0}(\Gamma) := \{h:V \rightarrow \mathbb Z \mid h(v_0)=0 \text{ and } \{u, v\} \in E \implies |h(u)-h(v)|=1\}.
\]
We will call an element of $\Hom_{v_0}(\Gamma)$ a \textit{rooted $\mathbb Z$-homomorphism} from $\Gamma$ (with respect to $v_0$). Define the \textit{range} of $h \in \Hom_{v_0}(\Gamma)$ as
\[R(h)=|\{h(x):x \in V\}|.\]

\subsection{Main results}

Our first main result, \Cref{thm:Z-hom-spectral-expanders}, concerns the typical range of $\mathbb{Z}$-homomorphisms on bipartite spectral expanders, and is motivated by a question raised by Peled, Samotij, and Yehudayoff \cite{peled2013lipschitz}. As a slight variant of standard terminology, we define a \textit{bipartite $(n,d,\lambda)$-graph} to be a $d$-regular bipartite graph with $n$ vertices in each part, such that its second largest eigenvalue is at most $\lambda$.

In \cite{peled2013lipschitz}, Peled, Samotij, and Yehudayoff investigated the typical range of integer-valued Lipschitz functions and $\mathbb Z$-homomorphisms on expanders. Focusing on their results for $\mathbb{Z}$-homomorphisms, we note that they proved that if the underlying graph is a sufficiently strong expander---meaning the parameter $\lambda$ is much smaller than the degree $d$, specifically $\lambda \le \frac{d}{300\log d}$---then a typical $\mathbb{Z}$-homomorphism takes only $O \left( \frac{\log \log n}{\log (d/\lambda)} \right)$ values with high probability.\footnote{To be precise, the notion of expanders used in \cite{peled2013lipschitz} was defined via the Expander Mixing Lemma (EML) \cite{alon1988explicit}: they defined a $d$-regular bipartite graph $\Gamma$ with vertex bipartition $\cE \cup \cO$ such that $|\cE|=|\cO|=n$ to be a \textit{$\lambda$-expander} if $\left|e(S,T)-\frac{d}{n}|S||T|\right|\le\lambda \sqrt{|S||T|} \quad \text{for all $S \subseteq \cE, T \subseteq \cO$}$. While the EML condition is sufficient in their strong expansion regime ($\lambda \ll d$), our focus on the much weaker expansion regime (where $\lambda$ is close to $d$) necessitates the standard algebraic definition of (bipartite) $(n,d,\lambda)$-graphs. See \Cref{rmk:expander} for more explanation.} 

To understand this bound, it is helpful to compare it with the trivial upper bound. It is a standard fact that the diameter of a bipartite $(n,d,\lambda)$-graph is $O\big( \frac{\log n}{\log (d/\lambda)} \big)$. Since the values of a $\mathbb{Z}$-homomorphism change by exactly $1$ along any edge, the range of \textit{any} $\mathbb{Z}$-homomorphism is deterministically bounded by the diameter of the graph (plus one). Thus, the result of \cite{peled2013lipschitz} shows a strong ``flatness'' phenomenon: the typical range is double-logarithmic in $n$, which is much smaller than the logarithmic diameter. Furthermore, this upper bound is tight: an earlier result by Benjamini, Yadin, and Yehudayoff \cite{benjamini2007random} established a matching lower bound, proving that the typical range is indeed $\Omega\big( \frac{\log \log n}{\log (d/\lambda)} \big)$.

Following their result, Peled, Samotij, and Yehudayoff naturally asked whether a similar flatness phenomenon continues to hold for graphs with a much smaller spectral gap (expansion), such as when $\lambda = (1-\delta)d$ for small $\delta > 0$. Our first main theorem answers this question affirmatively:

\begin{thm}\label{thm:Z-hom-spectral-expanders}
For every $\delta>0$, there exist constants $d_0=d_0(\delta)>0$ and $C=C(\delta)>0$ such that the following holds. Let $\Gamma$ be a bipartite $(n,d,\lambda)$-graph with vertex bipartition $V(\Gamma)=\cE\cup \cO$ such that $d \ge d_0$ and $\lambda\le (1-\delta)d$. Fix any $v_0\in \cE$. Let $\mathbf h$ be a random $\mathbb Z$-homomorphism chosen uniformly from $\Hom_{v_0}(\Gamma)$. Then, with probability $1-o_n(1)$,
\[
R(\mathbf h)\leq C\log\log n.
\]
\end{thm}

In \Cref{thm:Z-hom-spectral-expanders}, we state the result for a fixed constant $\delta>0$ for clarity of presentation. However, our argument still holds even if $\delta$ decays as a function of $d$ (i.e., $\delta = o_d(1)$). As a concrete application of our framework in the $\delta = o_d(1)$ regime, our second main result, \Cref{thm:Z-hom-middle-layers}, shifts the focus to a specific combinatorial structure: the middle layers of the Hamming cube. 

Our investigation in this setting is naturally motivated by two related lines of research. The first motivation is given by the study of the ``flatness'' of typical $\mathbb{Z}$-homomorphisms on the Hamming cube, $Q_d$. While the maximum possible range of a $\mathbb{Z}$-homomorphism on $Q_d$ is trivially $d+1$ (its diameter plus one), it was initially conjectured in \cite{benjamini2000random} that the typical range is sublinear, $o(d)$. Kahn \cite{kahn2001range} resolved this in a strong form, proving that the typical range is actually $O(1)$. Galvin \cite{galvin2003homomorphisms} later improved this constant, demonstrating that a random $\mathbb{Z}$-homomorphism on $Q_d$ takes at most $5$ values with high probability. These results---later vastly generalized to high-dimensional tori by Peled \cite{peled2017high}---show an extreme rigidity in the typical behavior of $\mathbb{Z}$-homomorphisms. Kahn's proof makes crucial use of the sophisticated structure of the Hamming cube, and Galvin's subsequent refinement builds on Kahn's earlier work.

Given the phenomenon explored by Peled, Samotij, and Yehudayoff \cite{peled2013lipschitz}, a natural question arises: is this extreme strong flatness (taking only at most 5 values) a consequence of the rigid structure of the Hamming cube, or does it still hold in more general settings? To answer this, we turn to a natural testbed: the middle layers of the Hamming cube. While this specific subgraph has attracted independent interest (e.g. \cite{balogh2025maximal,balogh2021independent,  duffus2013maximal, ilinca2013counting, li2023number}), for our present purposes, it provides a concrete setting to explore how much we can weaken the structural requirements.

Before formally stating our result, we introduce some relevant definitions. Let $Q_{2d-1}$ denote the $(2d-1)$-dimensional Hamming cube. For a vertex $v \in V(Q_{2d-1}) = \{0,1\}^{2d-1}$, let $\|v\|$ denote its Hamming weight (the number of 1's in the coordinates of $v$), and let
\[
    L_{k} := \{v\in V(Q_{2d-1}):\|v\|=k\}
\]
denote the \textit{$k$-th layer} of $Q_{2d-1}$. Let $Q_{2d-1}^{\md} := Q_{2d-1}[L_{d-1}\cup L_d]$ denote the induced subgraph of $Q_{2d-1}$ on the two middle layers. Observe that $Q_{2d-1}^{\md}$ is a balanced bipartite $d$-regular graph with parts $L_{d-1}$ and $L_d$. 
It is a known fact (see, e.g., \cite[Chapter 2]{stanley2013algebraic}) that the second largest eigenvalue of $Q_{2d-1}^{\md}$ is $\lambda = d - \Theta(1)$. This places $Q_{2d-1}^{\md}$ in the $\delta = \Theta(1/d) ~(= o_d(1))$ regime of \Cref{thm:Z-hom-spectral-expanders}.

Our second main theorem shows that the strong $O(1)$-flatness for the full Hamming cube still holds in this less rigid structure.

\begin{thm}\label{thm:Z-hom-middle-layers}
Fix an arbitrary vertex $v_0\in L_{d-1}$. Let $\mathbf{h}$ be a random $\mathbb Z$-homomorphism chosen uniformly from $\Hom_{v_0}(Q_{2d-1}^{\md})$. Then, with probability $1-o_d(1)$,
\[
    R(\mathbf{h}) \le 5.
\]
\end{thm}

\noindent We remark that our proof of \Cref{thm:Z-hom-middle-layers} primarily relies on the isoperimetric properties of the middle layers---following the approach in the proof of \Cref{thm:Z-hom-spectral-expanders}---rather than their exact combinatorial structure. Because the expansion of the middle layers is strictly weaker than that of the full Hamming cube, our method naturally recovers the results of Kahn \cite{kahn2001range} and Galvin \cite{galvin2003homomorphisms} for the full cube, again without relying on its specific structure.\footnote{Galvin \cite{galvin2003homomorphisms} actually provided a much more detailed description of the distribution of the range. Furthermore, one could potentially obtain even more precise probabilistic information by employing the cluster expansion method of Jenssen and Perkins \cite{jenssen2020independent}. While our current approach can be extended to yield similar distributional results, we do not pursue this direction here for the sake of simplicity and to avoid excessive repetition.}

The second motivation for studying $\mathbb{Z}$-homomorphisms on the middle layers stems from the study of graph homomorphisms on these layers. In this direction, Balogh, Garcia and Li \cite{balogh2021independent} determined the asymptotic number and typical structure of independent sets on the middle layers. Following their work, Li, McKinley, and the second author \cite{li2023number} determined the typical structure of $q$-colorings; interestingly, their method only applies when $q$ is even, leaving the problem completely open even for $3$-colorings, the smallest odd case.

This gap in the literature makes the study of $\mathbb{Z}$-homomorphisms on the middle layers particularly interesting. On the full Hamming cube, there is a natural mapping---reducing values modulo $3$---between $\mathbb{Z}$-homomorphisms and proper $3$-colorings; this key equivalence, attributed to Dana Randall in \cite{galvin2003homomorphisms}, allowed Galvin to immediately deduce the typical structure of $3$-colorings from his result on $\mathbb Z$-homomorphisms. However, on the middle layers, this connection simply fails. Consequently, without this bijection and with $3$-colorings remaining out of reach, determining the typical behavior of $\mathbb{Z}$-homomorphisms on the middle layers remains an interesting problem on its own.

\subsection{Proof overview and key ideas}

Historically, a well-established framework for proving typicality results for graph homomorphisms and related discrete models (such as ordinary/maximal independent sets, proper colorings and Lipschitz functions) on bipartite (weak) expanders relies heavily on a two-step ``ground state'' paradigm (e.g., \cite{balogh2025maximal,balogh2021independent,galvin2003homomorphisms, jenssen2023homomorphisms,JMP1,jenssen2020independent,   kahn2020number, kahn2022number,    krueger2024lipschitz,   li2023number,peled2013lipschitz,sapoposet,sapo}). This framework first establishes that a typical configuration globally aligns with a rigid ground state (e.g., for $\mathbb{Z}$-homomorphisms, even vertices take $0$ and odd vertices take $\pm 1$). The second step then employs a Peierls-type contour argument (see, e.g., \cite{peierls1936ising, korshunov1983number}) to bound the connected components of ``defects'' deviating from this ground state, showing that large deviations result in a large entropic penalty.

Peled, Samotij, and Yehudayoff \cite{peled2013lipschitz} also used this framework. Their starting point was a key structural lemma that establishes a global ground state. They proved that for any bipartite $(n,d,\lambda)$-graph on $(V_0, V_1)$, every $\mathbb{Z}$-homomorphism $f$ mostly takes a single value on one side. That is, there exist $i \in \{0,1\}$ and $k \in \mathbb{Z}$ such that
\begin{equation}\label{eq:psy-ground-state}
    |\{v \in V_i: f(v) \ne k\}| \le \frac{2\lambda n}{d}.
\end{equation}
Recently, Krueger, Li, and the second author \cite{krueger2024lipschitz} pushed the expansion requirement for the flatness of typical \textit{$M$-Lipschitz functions} to $\lambda < d/5$, improving upon the bound $O(d/(M\log(Md^2)))$ in \cite{peled2013lipschitz}. While they noted that their approach could be adapted to yield similar results for $\mathbb{Z}$-homomorphisms, \eqref{eq:psy-ground-state} remains a fundamental bottleneck. For the inequality in \eqref{eq:psy-ground-state} to provide a meaningful restriction---specifically, to ensure that a majority of vertices in $V_i$ actually take the value $k$---we must have $\frac{2\lambda n}{d} < n$, which requires $\lambda < d/2$. Consequently, this key lemma does not provide a useful bound for weak expanders, where $\lambda$ can be as large as $(1-\delta)d$. To bypass this limitation, our proof does not rely on establishing a global ground state.

Instead, we introduce a local smoothing procedure called the \textit{merge} operation. Rather than comparing a configuration to a single rigid ground state, we analyze how the number of valid configurations changes as we ``flatten'' local extrema. Specifically, we demonstrate that applying the merge operation to a set of $\mathbb{Z}$-homomorphisms with large peaks maps them into a space of flatter configurations that is significantly larger. To carefully bound the size of the preimage of this operation, we employ the graph container method, a powerful tool in this area. While our container construction closely follows  \cite{galvin2003homomorphisms}, the key to our proof is the combination of these containers with the newly introduced merge operation. As a consequence, configurations with tightly bounded fluctuations constitute almost all of the set of valid configurations, allowing us to deduce the typical structure directly without establishing a global ground state. Crucially, by avoiding this step, our argument no longer requires any special structural properties or strong expansion assumptions on the underlying graph.

\begin{remark}\label{rmk:expander}
As noted in an earlier footnote, our definition of an expander differs from the one used in  \cite{peled2013lipschitz}. They defined a bipartite expander based on the conclusion of the Expander Mixing Lemma \cite{alon1988explicit}. This specific formulation yields the following vertex expansion (see Proposition 2.8 in \cite{peled2013lipschitz}): for every subset $A \subseteq V(G)$,
\[
    |N(A)| \ge \min \left\{ \frac{n}{2}, \frac{d^2}{4\lambda^2}|A| \right\}.
\]
Crucially, this bound does not provide an expansion guarantee when $\lambda>d/2$. To avoid this technical limitation, we adopt the standard spectral definition of a bipartite $(n,d,\lambda)$-graph. Since every $(n,d,\lambda)$-graph satisfies the conclusion of the Expander Mixing Lemma, the assumption in \cite{peled2013lipschitz} is strictly weaker than ours. However, we emphasize that adopting this stronger definition is not merely to bypass the bottleneck encountered in \cite{krueger2024lipschitz, peled2013lipschitz}. The fundamental obstacle regarding \eqref{eq:psy-ground-state} applies equally under both definitions; our resolution relies entirely on the introduction of the merge operation.
\end{remark}

\subsection*{Definitions and notation} 

Given a graph $\Gamma$, let $V(\Gamma)$ and $E(\Gamma)$ denote its vertex and edge sets, respectively. For disjoint sets of vertices $A$ and $B$, we denote by $E(A, B)$ the set of edges with one endpoint in $A$ and the other in $B$. For a vertex $v \in V(\Gamma)$ and subsets $A, F \subseteq V(\Gamma)$, we define
\[
    N_F(v) := \{w \in F : \{v,w\} \in E(\Gamma)\}, \qquad N_F(A) := \bigcup_{v \in A} N_F(v),
\]
and let $d_F(v) := |N_F(v)|$ and $d_F(A) := |N_F(A)|$. When the underlying graph is clear from the context, we drop the subscript and simply write $N(\cdot)$ for $N_{V(\Gamma)}(\cdot)$ and $d(\cdot)$ for $d_{V(\Gamma)}(\cdot)$.
We define the second neighborhood as $N^2(v) := N(N(v))$, and similarly, $N^2(A) := \bigcup_{v \in A} N^2(v)$. (Note that under this definition, $N^2(v)$ may include $v$ itself, and similarly, $N^2(A)$ may intersect $A$.)

The distance $\dist_\Gamma(u,v)$ (or $\dist(u,v)$ when $\Gamma$ is clear) is the length of a shortest path between $u$ and $v$ in $\Gamma$. The diameter of $\Gamma$ is defined as
\[
    \diam(\Gamma) := \max_{u,v \in V(\Gamma)} \dist(u,v).
\]
Throughout this paper, all logarithms are base $2$, i.e., $\log x := \log_2 x$. For a positive integer $k$, we denote the set $\{1, 2, \dots, k\}$ by $[k]$. We omit floor and ceiling signs and assume that all sufficiently large quantities are integers where appropriate, as is common practice.

We always assume that the parameter $d$ is sufficiently large. Unless stated otherwise, all asymptotic notation ($O$, $\Omega$, $\Theta$, $o$, $\omega$) is with respect to $d \rightarrow \infty$. We write $f \ll g$ to denote $f=o(g)$, and $f \gg g$ to denote $f=\omega(g)$. Finally, we write $f|_S$ for the restriction of a function $f$ to a subset $S$.

\subsection*{Organization of the paper}
The remainder of this paper is organized as follows. In \Cref{sec:prelim}, we collect basic preliminaries. In \Cref{sec:containers}, we introduce the graph container lemmas. In Sections \ref{sec:reduction} and \ref{sec:merge}, we present the merge operation. Finally, in Sections \ref{sec:proof-spectral} and \ref{sec:proof-middle layers}, we prove Theorems \ref{thm:Z-hom-spectral-expanders} and \ref{thm:Z-hom-middle-layers}, respectively.

\section{Preliminaries}\label{sec:prelim}

\subsection{Graph theory basics}

\begin{mydef}[$k$-linked set and component]
We say a subset $A\subseteq V(\Gamma)$ is \emph{$k$-linked} in a graph $\Gamma$ if for every pair of vertices $u,v\in A$, there exists a sequence of vertices $u=v_0,v_1,\dots,v_m=v$ in $A$ such that $\dist_\Gamma(v_i,v_{i+1})\le k$ for all $0\le i<m$. 
Furthermore, for a given subset $S \subseteq V(\Gamma)$, a \emph{$k$-linked component} of $S$ is defined as a maximal $k$-linked subset of $S$.
\end{mydef}

The following lemma follows from the standard fact (see, e.g., \cite[p.~396, Ex.~11]{knuth1968art}) that the infinite $\Delta$-branching rooted tree contains at most $(e\Delta)^{t-1}$ rooted subtrees with $t$ vertices.

\begin{lem}\label{lem:linked-set-count}
    Let $\Gamma$ be a graph with maximum degree $\Delta$. For every fixed integer $k\geq 1$, the number of $k$-linked subsets of $V(\Gamma)$ of size $t$ that contain a given vertex is at most $(e\Delta^k)^{t-1}$.
\end{lem}

\begin{mydef}\label{def:bipartite-cover}
Let $\Gamma$ be a bipartite graph with vertex bipartition $V(\Gamma)=\cE\cup \cO$. For $P\subseteq \cE$ and $Q\subseteq \cO$, we say that $Q$ \emph{covers} $P$ if $P\subseteq N(Q)$.
\end{mydef}

\begin{lem}[Lov\'asz \cite{Lov75}, Stein \cite{Ste74}] \label{lem:cover}
Let $\Gamma$ be a bipartite graph with vertex bipartition $V(\Gamma)=\cE\cup \cO$. Assume that
\[
|N(x)|\ge a \quad\text{for every }x\in \cE, \qquad \text{and} \qquad |N(y)|\le b \quad\text{for every }y\in \cO.
\]
Then, $\cE$ is covered by some subset $Q\subseteq \cO$ satisfying $|Q|\le \frac{|\cO|}{a}(1+\ln b)$.
\end{lem}

\subsection{Isoperimetry in bipartite spectral expanders}

We begin by stating a vertex expansion bound for bipartite graphs based on their second largest eigenvalue. The following lemma is a bipartite variant of the classical Tanner bound \cite{tanner1984explicit} (for a broader overview of such spectral techniques, we refer the reader to \cite{hoory2006expander}).

\begin{lem}\label{prop:tanner}
Let $\Gamma=(V,E)$ be a bipartite $(n,d,\lambda)$-graph with vertex bipartition $V=\cE\cup \cO$, such that $\lambda=\alpha d$.
If $S\subseteq \cE$ (or $S\subseteq \cO$) and $|S|=\rho n$, then we have
\[
\frac{|N(S)|}{|S|}
\ge
\frac{1}{\rho(1-\alpha^2)+\alpha^2}.
\]
\end{lem}

\begin{proof}
Let $A$ be the $2n \times 2n$ adjacency matrix of $\Gamma$. Since $\Gamma$ is a bipartite $d$-regular graph, its eigenvalues occur in pairs $(\mu,-\mu)$, with top eigenvalues $d$ and $-d$, and all remaining eigenvalues of absolute value at most $\lambda=\alpha d$.

Choose an orthonormal eigenbasis $v_1,\dots,v_{2n} \in \mathbb R^{\cE\cup \cO}$ so that
\[
v_1=\frac{1}{\sqrt{2n}}\mathbf 1_{\cE\cup \cO},
\qquad
v_{2n}=\frac{1}{\sqrt{2n}}(\mathbf 1_{\cE}-\mathbf 1_{\cO}),
\]
with eigenvalues $d$ and $-d$, respectively. Write $\mathbf 1_S:=\sum_{i=1}^{2n} a_i v_i$. 
Since $S\subseteq \cE$ and $|S|=\rho n$, we have
\[
a_1=\mathbf 1_S\cdot v_1=\rho\sqrt{\frac{n}{2}},
\qquad
a_{2n}=\mathbf 1_S\cdot v_{2n}=\rho\sqrt{\frac{n}{2}}.
\]
Therefore,
\[
\|A\mathbf 1_S\|_2^2
=
\sum_{i=1}^{2n}\lambda_i^2 a_i^2
=
d^2\rho^2 n+\sum_{i=2}^{2n-1}\lambda_i^2 a_i^2
\le
d^2\rho^2 n+\alpha^2 d^2\sum_{i=2}^{2n-1} a_i^2.
\]
Now,
\[
\sum_{i=2}^{2n-1}a_i^2
=
\|\mathbf 1_S\|_2^2-a_1^2-a_{2n}^2
=
|S|-\rho^2 n
=
\rho n-\rho^2 n,
\]
so
\begin{align}\label{eq:tanner-1}
    \|A\mathbf 1_S\|_2^2
\le
d^2\rho^2 n+\alpha^2 d^2(\rho n-\rho^2 n)
=
d^2 |S|(\rho+\alpha^2(1-\rho)).
\end{align}

On the other hand, since $\Gamma$ is bipartite and $S\subseteq \cE$, the vector $A\mathbf 1_S$ has entries equal to zero on $\cE$. For every $y\in \cO$, we have $(A\mathbf 1_S)_y=d_S(y)$,
so $A\mathbf 1_S$ is exactly supported on $N(S)$, with
\[
\sum_{y\in N(S)} (A\mathbf 1_S)_y=\sum_{y\in N(S)}d_S(y)=d|S|.
\]
By Cauchy--Schwarz, we have
\begin{align}\label{eq:tanner-2}
    d^2|S|^2
=
\Bigl(\sum_{y\in N(S)} (A\mathbf 1_S)_y\Bigr)^2
\le
|N(S)|\cdot \|A\mathbf 1_S\|_2^2.
\end{align}
Combining \eqref{eq:tanner-1} and \eqref{eq:tanner-2} gives
\[
d^2|S|^2
\le
|N(S)|\cdot d^2|S|(\rho+\alpha^2(1-\rho)),
\]
and hence
\[
\frac{|N(S)|}{|S|}
\ge
\frac{1}{\rho+\alpha^2(1-\rho)}
=
\frac{1}{\rho(1-\alpha^2)+\alpha^2}.
\qedhere
\]
\end{proof}

\begin{cor}\label{prop:expand-3/4}
Fix $\delta>0$. Let $\Gamma=(V,E)$ be a bipartite $(n,d,\lambda)$-graph with vertex bipartition $V=\cE\cup \cO$, such that $\lambda\le (1-\delta)d$. Then every set $S\subseteq \cE$ (or $S\subseteq \cO$) with $|S|\le \frac{3n}{4}$
satisfies
\[
|N(S)|\ge \Bigl(1+\frac{\delta}{4}\Bigr)|S|.
\]
\end{cor}

\begin{proof}
Write $\alpha=\lambda/d$, so $\alpha\le 1-\delta$. If $|S|=\rho n\le 3n/4$, then \Cref{prop:tanner} gives
\[
\frac{|N(S)|}{|S|}
\ge
\frac{1}{\rho(1-\alpha^2)+\alpha^2}
=
\frac{1}{1-(1-\alpha^2)(1-\rho)}.
\]
Using the inequality $\frac{1}{1-x} \ge 1+x$ for $x \in (0,1)$, and noting that $1-\rho\ge 1/4$ and $1-\alpha^2=(1-\alpha)(1+\alpha)\ge 1-\alpha\ge \delta$, we have
\[
\frac{|N(S)|}{|S|}
\ge
1+(1-\alpha^2)(1-\rho)
\ge
1+\frac{\delta}{4}. \qedhere
\]
\end{proof}

\subsection{Isoperimetry in $Q_{2d-1}^{\mathrm{mid}}$}

We will utilize the symmetry between $L_{d-1}$ and $L_d$ in $Q_{2d-1}^{\mathrm{mid}}$.

\begin{obs}\label{obs:symmetry}
    The graph $Q_{2d-1}^{\mathrm{mid}}$ has an automorphism that swaps $L_{d-1}$ and $L_d$.
\end{obs}
\begin{proof}
Define $\tau:L_{d-1}\cup L_d\to L_{d-1}\cup L_d$ by $\tau(x)=\mathbf 1-x$, i.e., by flipping every coordinate. Then $\tau$ preserves adjacency in $Q_{2d-1}^{\mathrm{mid}}$, with $\tau(L_{d-1})=L_d$ and $\tau(L_d)=L_{d-1}$.
\end{proof}

For $x \in \RR$, write $\binom{x}{i}:=\frac{x(x-1)\ldots(x-i+1)}{i!}$. We recall the Lov\'asz version of the Kruskal-Katona theorem \cite{kruskal1963number, katona1968theorem}. 

\begin{thm}[Lov\'asz \cite{lovasz2007combinatorial}] \label{thm:lovasz}
Let $\cA$ be a family of $m$-element subsets of a fixed set $U$, and let $\cB$ be the family of all $(m-q)$-element subsets of the sets in $\cA$. If $|\cA|=\binom{x}{m}$ for some real number $x \geq m$, then $|\cB|\geq\binom{x}{m-q}$.    
\end{thm}

\begin{lem}\label{lem:isoperimetry}
    Let $Y\subseteq L_{d}$ (or $Y\subseteq L_{d-1}$). Let $N(Y)$ be the neighborhood of $Y$ in $Q_{2d-1}^{\mathrm{mid}}$.
\begin{enumerate}[(a)]
    \item If $|Y|\leq \binom{2d-1-c}{d}$ for some $c\geq 0$, then $|N(Y)|\geq \left(1+\frac{c}{d}\right)|Y|$.
    \item If $|Y|\leq d^6$, then $|N(Y)|\geq d|Y|/8$.
\end{enumerate}
\end{lem}

\begin{proof} 
By \Cref{obs:symmetry}, we can assume without loss of generality that $Y\subseteq L_d$. By naturally identifying the binary vectors in $L_d$ with $d$-element subsets of $[2d-1]$, we can view $Y$ as a family of subsets.

We first verify (a). Choose a real number $c_1\geq c$ so that $|Y|=\binom{2d-1-c_1}{d}$. By applying \Cref{thm:lovasz} with $U=[2d-1]$, $m=d$, $q=1$, and $\mathcal A=Y$, we obtain that
\begin{align*}
    |N(Y)| \geq \binom{2d-1-c_1}{d-1} = \frac{d}{d-c_1}\cdot |Y| \geq \left(1+\frac{c_1}{d}\right)|Y| \geq \left(1+\frac{c}{d}\right)|Y|.
\end{align*}   

To see (b), observe that when $|Y|\leq d^6$, we have $|Y|=\binom{2d-1-c_1}{d}$ for some real number $c_1\geq d-8$ (assuming $d$ is sufficiently large, since $\binom{2d-1-(d-8)}{d} = \binom{d+7}{7} \gg d^6$). Thus, we have 
\[
    |N(Y)| \geq \frac{d}{d-c_1}\cdot |Y| \geq \frac{d}{8}|Y|. \qedhere
\]
\end{proof}

\section{Graph containers}\label{sec:containers}

In this section, we introduce graph container lemmas, which we will use in subsequent sections to bound the number of specific vertex subsets in our bipartite graphs. While this technique was initially pioneered by Sapozhenko \cite{sapozhenko1987number} to enumerate independent sets in a certain class of bipartite graphs, the specific formulation we employ here is due to Galvin \cite{galvin2003homomorphisms} -- particularly, the introduction of the associated sets $B$ and $H$, as well as the use of approximating quadruples.

\subsection{The $\psi$-approximating quadruples}\label{sec:covering-approx}

Given a subset $A$, we define the following sets associated with $A$.

\begin{mydef}\label{def:agbh}
Let $\Gamma=(V, E)$ be a bipartite graph with vertex bipartition $V=\cE\cup\cO$.
For every $A\subseteq \cE$, define the following four sets of vertices associated with $A$:
\begin{align*}
    G=G(A)&:=N(A),\\
    [A]&:=\{v\in\cE:N(v)\subseteq N(A)\},\\
    B=B(A)&:=\{v\in \cO:N(v)\subseteq A\},\\
    H=H(A)&:=N(B).
\end{align*}
Let $g=|G|$, $a=|A|$, $a'=|[A]|$, $b=|B|$, and $h=|H|$. Let $t_1=g-a'$ and $t_2=h-b$.
\end{mydef}

Note that from the above definitions, we have $G=N([A])$ and $B=[B]$. Also, when $\Gamma$ is a $d$-regular graph, we always have $b\leq h\leq a\leq a'\leq g$ (since $|N(S)|\geq|S|$ for all $S\subseteq V(\Gamma)$). 

\begin{mydef}\label{def:families}
Let $\Gamma=(V, E)$ be a bipartite graph with vertex bipartition $V=\cE\cup\cO$. For integers $b\le h\le a'\le g$ and $v\in \cO$, define
\begin{align*}
\mathcal H(a',g,b,h,v)
&:=
\bigl\{
A\subseteq \cE:\ 
A\text{ is $2$-linked},\ |[A]|=a',\ |G|=g,\ |B|=b,\ |H|=h,\ v\in G
\bigr\},\\
\mathcal H(a',g,b,h)
&:=
\bigl\{
A\subseteq \cE:\ 
A\text{ is $2$-linked},\ |[A]|=a',\ |G|=g,\ |B|=b,\ |H|=h
\bigr\}.
\end{align*}
\end{mydef}

The following definition introduces a tuple of sets that closely approximates $A$ along with its associated sets from \Cref{def:agbh}.

\begin{mydef}[$\psi$-approximating quadruple]\label{def:psi-approx-quad}
Let $\Gamma=(V, E)$ be a $d$-regular bipartite graph with vertex bipartition $V=\cE \cup \cO$. For $0\leq\psi\leq d$, a \emph{$\psi$-approximating quadruple} for a set $A\subseteq \cE$ is a quadruple $(F,S,P,Q)\in 2^{\cO}\times 2^{\cE}\times 2^{\cE}\times 2^{\cO}$ such that:
\begin{enumerate}
\item $F\subseteq G$ and $S\supseteq [A]$,
\item $d_F(x)\ge d-\psi$ for every $x\in S$,
\item $d_{\cE\setminus S}(y)\ge d-\psi$ for every $y\in \cO\setminus F$,
\item $P\subseteq H$ and $Q\supseteq B$,
\item $d_P(y)\ge d-\psi$ for every $y\in Q$,
\item $d_{\cO\setminus Q}(x)\ge d-\psi$ for every $x\in \cE\setminus P$.
\end{enumerate}
\end{mydef}

\subsection{Container lemmas for approximating quadruples}

In this subsection, we state two results (\Cref{cor:H-spectral-expander} and \Cref{cor:H-middle-layers}) based on container-type arguments: for bipartite $(n,d,\lambda)$-graphs and for the middle layers of the Hamming cube, 
there exists a small family of tuples of sets that can serve as $\psi$-approximating quadruples 
for all sets $A\subseteq \cE$ with fixed parameters $a',g,b,h$ (and $v$). 
While similar container lemmas were established by Galvin \cite{galvin2003homomorphisms}, our host graphs differ slightly from those considered in his work. For completeness and ease of reference, we provide the full proofs adapted to our setting in the appendix. The ideas in the appendix closely follow \cite{galvin2003homomorphisms, galvin2019independent}.

Looking ahead, we remark that we will apply the results in this section with the choice of $\psi = \sqrt{d}$.

\begin{lem}\label{prop:approx-quad-spectral-expander}
Let $\Gamma$ be a bipartite $d$-regular graph with vertex bipartition $V(\Gamma)=\cE\cup \cO$.
Fix $b\le h\le a'\le g$ and $v\in \cO$. Then there exists a family
\[
\mathcal U_{\mathrm{quad}}(a',g,b,h,v)\subseteq 2^{\cO}\times 2^{\cE}\times 2^{\cE}\times 2^{\cO}
\]
with
\[
|\mathcal U_{\mathrm{quad}}(a',g,b,h,v)|
\le
2^{O(g\log d/\psi)},
\]
such that every $A\in \mathcal H(a',g,b,h,v)$ has a $\psi$-approximating quadruple in $\mathcal U_{\mathrm{quad}}(a',g,b,h,v)$.
\end{lem}

We prove \Cref{prop:approx-quad-spectral-expander} in \Cref{app:1}.

\begin{lem}\label{prop:approx-quad-middle-layers}
Let $\Gamma=Q_{2d-1}^{\mathrm{mid}}$ with $\cE=L_{d-1}$ and $\cO=L_{d}$.
Fix $b\le h\le a'\le g$, and assume $g\ge d^2/100$ and $a'\leq (1-\frac{1}{10})\binom{2d-1}{d}$.
Then there exists a family
\[
\mathcal U_{\mathrm{quad}}(a',g,b,h)\subseteq 2^{\cO}\times 2^{\cE}\times 2^{\cE}\times 2^{\cO}
\]
with
\[
|\mathcal U_{\mathrm{quad}}(a',g,b,h)|
\le
2^{O(t_1\log d/\psi)}\cdot 2^{O(t_2\log d/\psi)},
\]
such that every $A\in \mathcal H(a',g,b,h)$ has a $\psi$-approximating quadruple in $\mathcal U_{\mathrm{quad}}(a',g,b,h)$.
\end{lem}

We prove \Cref{prop:approx-quad-middle-layers} in \Cref{app:2}.

We also state a ``reconstruction'' lemma, which shows that only few sets $A$ share the same $\psi$-approximating quadruple.

\begin{lem}\label{lem:reconstruction}
Let $\Gamma$ be a bipartite $d$-regular graph with vertex bipartition $V(\Gamma)=\cE\cup \cO$. Fix $b\le h\le a'\le g$, and assume that
\beq{eq:t}
t_1=\Omega(g/d),
\qquad
t_2=\Omega(h/d).
\enq
 Then for every quadruple $(F,S,P,Q)$ satisfying conditions (2), (3), (5), and (6) of \Cref{def:psi-approx-quad} with $\psi=\sqrt d$, the number of $2$-linked sets $A\subseteq \cE$ with
\[
|[A]|=a',
\qquad
|G|=g,
\qquad
|B|=b,
\qquad
|H|=h,
\]
for which $(F,S,P,Q)$ also satisfies conditions (1) and (4) of \Cref{def:psi-approx-quad} is at most
\[
2^{\,g-b- \Omega(t_1/\log d)- \Omega(t_2/\log d)}.
\]
\end{lem}

We prove \Cref{lem:reconstruction} in \Cref{app:3}.

Combining \Cref{prop:approx-quad-spectral-expander} and \Cref{prop:approx-quad-middle-layers} with \Cref{lem:reconstruction}, we obtain the following results.

\begin{cor}\label{cor:H-spectral-expander}
Fix $\delta > 0$. Let $\Gamma$ be a bipartite $(n,d,\lambda)$-graph with vertex bipartition $V(\Gamma)=\cE\cup \cO$, such that $\lambda\le (1-\delta)d$.
Fix $b\le h\le a'\le g$ and $v\in \cO$, and assume $a'\le 3n/4$.
Then
\[
|\mathcal H(a',g,b,h,v)|
\le
2^{\,g-b-\Omega(t_1/\log d)}.
\]
\end{cor}

\begin{proof}
    We know from \Cref{prop:expand-3/4} that if $a'\leq 3n/4$, then $g\geq (1+\delta/4)a'$, so $t_1=g-a'\geq (1-\frac{1}{1+\delta/4})g\geq   (\delta/8)g$. In particular, for any fixed constant $\delta>0$ we have $g\log d/\psi\ll t_1/\log d$. By \Cref{prop:approx-quad-spectral-expander} and \Cref{lem:reconstruction}, we obtain that
\begin{align*}
       |\mathcal H(a',g,b,h,v)|&\leq |\mathcal U_{\mathrm{quad}}(a',g,b,h,v)|\cdot 2^{\,g-b- \Omega(t_1/\log d)- \Omega(t_2/\log d)}\\
       &=2^{O(g\log d/\psi)}\cdot 2^{\,g-b- \Omega(t_1/\log d)- \Omega(t_2/\log d)}\\
       &=2^{\,g-b-\Omega(t_1/\log d)}. \qedhere
   \end{align*}
\end{proof}

\begin{cor}\label{cor:H-middle-layers}
Let $\Gamma=Q_{2d-1}^{\mathrm{mid}}$ with $\cE=L_{d-1}$ and $\cO=L_{d}$.
Fix $b\le h\le a'\le g$, and assume $g\ge d^2/100$ and $a'\leq (1-\frac{1}{10})\binom{2d-1}{d}$.
Then
\[
|\mathcal H(a',g,b,h)|
\le
2^{\,g-b-\Omega(t_1/\log d)}.
\]
\end{cor}
\begin{proof}
    Note that $Q_{2d-1}^{\mathrm{mid}}$ satisfies \eqref{eq:t} by \Cref{lem:isoperimetry} (a). Therefore, by  \Cref{prop:approx-quad-middle-layers} and \Cref{lem:reconstruction},
   \begin{align*}
       |\mathcal H(a',g,b,h)|&\leq |\mathcal U_{\mathrm{quad}}(a',g,b,h)|\cdot 2^{\,g-b- \Omega(t_1/\log d)- \Omega(t_2/\log d)}\\
       &=2^{O(t_1\log d/\psi)}\cdot 2^{O(t_2\log d/\psi)}\cdot 2^{\,g-b- \Omega(t_1/\log d)- \Omega(t_2/\log d)}\\
       &=2^{\,g-b-\Omega(t_1/\log d)}. \qedhere
   \end{align*}
\end{proof}

\section{Reduction: from homomorphisms to partitions}\label{sec:reduction}

Let $\Gamma=(V, E)$ be a bipartite graph with vertex bipartition $V=\cE\cup\cO$, and fix an arbitrary vertex $v_0\in\cE$. For any $h\in\Hom_{v_0}(\Gamma)$ and any two vertices $u, v \in \cE$ at distance 2, their values under $h$ differ by exactly 0 or 2. By rescaling these values, we obtain a structure we call a \textit{legal labeling} of $\cE$.

\begin{mydef}[Legal labeling]\label{def:legal-labeling}
Let $\Gamma=(V, E)$ be a bipartite graph with vertex bipartition $V=\cE\cup\cO$, and $v_0 \in \cE$.
A \textit{legal labeling} of $\cE$ is a map $f:\cE\to \mathbb{Z}$ such that for every pair $u,v\in \cE$ with $\dist(u,v)=2$, we have $|f(u)-f(v)|\leq 1$. We say that a legal labeling $f$ is \textit{centered} (with respect to $v_0$) if $f(v_0)=0$.
\end{mydef}

For every centered legal labeling $f$ of $\cE$, we define its \textit{image} as 
\[
\mathrm{im}(f):=\{x \in \mathbb{Z} : \exists v \in \cE \text{ such that } f(v)=x\},
\]
noting that $|\mathrm{im}(f)|\le \mathrm{diam}(\Gamma)/2+1$. For each $i\in\mathrm{im}(f)$, the \textit{preimage} of $i$ under $f$ is $f^{-1}(i):=\{v\in \cE:f(v)=i\}$. Note that the level sets $\{f^{-1}(i)\}_{i \in \mathbb{Z}}$ naturally partition the vertex set $\cE$.

Recall the set $B(A)$ from \Cref{def:agbh}:
\[ B(A) := \{v \in \cO : N(v) \subseteq A\}. \]
A key observation, inspired by the approaches in \cite{galvin2003homomorphisms, peled2013lipschitz}, is that there is a one-to-many correspondence between centered legal labelings and $\Hom_{v_0}(\Gamma)$. 

\begin{lem}\label{obs:centered}
For every centered $\mathbb{Z}$-homomorphism $h:\cE\cup \cO \to \mathbb{Z}$, the rescaled restriction $f := \frac{1}{2}h|_{\cE}$ is a centered legal labeling of $\cE$. Conversely, for any centered legal labeling $f:\cE\to\mathbb{Z}$, the number of centered $\mathbb{Z}$-homomorphisms $h:\cE\cup \cO\to \mathbb{Z}$ that satisfy $h|_{\cE}=2f$ is exactly
\beq{def:P(f)}
P(f):=\prod_{i\in\mathrm{im}(f)}2^{|B(f^{-1}(i))|}.
\enq
We call $P(f)$ the \emph{weight} of $f$.
\end{lem}

\begin{proof}
    If $h(v_0)=0$, then $f(v_0)=0$, so $f$ is centered. For any $u, w \in \cE$ with $\dist(u, w)=2$, there exists a common neighbor $v \in \cO$. Since $|h(u)-h(v)|=1$ and $|h(w)-h(v)|=1$, we must have $|h(u)-h(w)| \in \{0, 2\}$. Dividing by 2 shows $|f(u)-f(w)| \le 1$, so $f$ is a centered legal labeling.

    For the converse, fix a centered legal labeling $f$. This determines $h|_\cE$, and consider the possible values for $h(v)$ for each $v \in \cO$. For $h$ to be a homomorphism with $h|_{\cE}=2f$, $h(v)$ must be adjacent to every value in $\{2f(u) : u \in N(v)\}$. 
    \begin{itemize}
        \item If $f$ is constant $i$ on $N(v)$ (i.e., $N(v) \subseteq f^{-1}(i)$; equivalently, $v \in B(f^{-1}(i))$), then $h(N(v)) = \{2i\}$. In this case, $h(v)$ can be either $2i-1$ or $2i+1$, giving $2$ choices.
        \item If $f$ is not constant on $N(v)$, then since $|f(u)-f(w)| \le 1$ for all $u, w \in N(v)$, the set $f(N(v))$ must be $\{i, i+1\}$ for some $i$. Then $h(N(v)) = \{2i, 2i+2\}$. The only integer adjacent to both values is $2i+1$, giving exactly $1$ choice.
    \end{itemize}
    Since the choices for each $v \in \cO$ are independent, the total number of homomorphisms is $2^{\sum_{i} |B(f^{-1}(i))|}$, which equals $P(f)$.
\end{proof}

Let $\mathcal{F}$ denote the set of all centered legal labelings of $\cE$. It follows from \Cref{obs:centered} that $|\Hom_{v_0}(\Gamma)|$ can be expressed as the total weight of these labelings:
\beq{eq:task} 
|\Hom_{v_0}(\Gamma)| = \sum_{f \in \mathcal{F}} P(f).
\enq

\section{Merge}\label{sec:merge}

The goal of this section is to introduce the \textit{merge} operation as a crucial tool to bound the sum in \eqref{eq:task}.

\subsection{The merge operation}

The purpose of the merge operation is to smooth out some of the ``peaks'' in a given legal labeling, which we call \textit{maximum components}.

\begin{mydef}[Maximum component]\label{def:max-component}
Let $\Gamma=(V, E)$ be a bipartite graph with vertex bipartition $V=\cE\cup\cO$.
Let $f:\cE\to\mathbb{Z}$ be a legal labeling of $\cE$. We say that $K\subseteq \cE$ is a \textit{maximum component} (of $f$) if:
\begin{enumerate}[(i)]
    \item $K$ is 2-linked; and
    \item for every $x\in K$ and $y\in N^2(K)\setminus K$, we have $f(y)<f(x)$.
\end{enumerate}
\end{mydef}

Observe that if $K$ is a maximum component of $f$, then for every $y \in N^2(K) \setminus K$, $f(y)=\min\{f(x):x \in K\}-1$.

\begin{mydef}[Merge] 
Let $f:\cE\to\mathbb{Z}$ be a centered legal labeling. Suppose there is a maximum component $K\subseteq \cE$ of $f$. Define the \textit{merge} (with respect to $f$ and $K$), denoted by $m_K(f):\cE\to\mathbb{Z}$, as follows:
\begin{itemize}
    \item If $v_0\notin K$, then let
    \[
    m_K(f)(x)=\begin{cases}f(x)-1 & \text{if } x \in K;\\ f(x) & \text{if } x \notin K.\end{cases}
    \]
    \item If $v_0\in K$, then let
    \[
    m_K(f)(x)=\begin{cases}f(x) & \text{if } x \in K;\\ f(x)+1 & \text{if } x \notin K.\end{cases}
    \]
\end{itemize}
\end{mydef}

\begin{lem}\label{obs:merge-preserve-labeling}
If $f$ is a centered legal labeling and $K$ is a maximum component of $f$, then $m_K(f)$ is also a centered legal labeling. Furthermore, we have $P(m_K(f))=2^{\,|N(K)|-|B(K)|}\,P(f)$.
\end{lem}

\begin{proof} 
It is clear that $m_K(f)(v_0)=f(v_0)=0$. To verify that $m_K(f)$ is a legal labeling, we must show that for every pair of vertices $u,w\in \cE$ with $\dist(u,w)=2$, $|m_K(f)(u)-m_K(f)(w)|\leq 1$. 

For simplicity, write $\widetilde{f}$ for $m_K(f)$. We consider three different cases up to symmetry. If $u,w\in K$ or $u,w\notin K$, it is clear from the definition that $|\widetilde{f}(u)-\widetilde{f}(w)|\leq 1$. If $u\in K$ and $w\notin K$, then since $w\in N^2(K)\setminus K$, we must have $f(w)=f(u)-1$ (by the definition of a legal labeling and the observation following \Cref{def:max-component}). Therefore, if $v_0\notin K$, we have $\widetilde{f}(w)=f(w)=f(u)-1=\widetilde{f}(u)$; if $v_0\in K$, we have $\widetilde{f}(w)=f(w)+1=f(u)=\widetilde{f}(u)$.

To prove that $P(\widetilde{f})=2^{\,|N(K)|-|B(K)|}\,P(f)$, we first express the weight $P(f)$ in a more combinatorial form. By \Cref{obs:centered},
\[
P(f) =\prod_{i\in\mathrm{im}(f)}2^{|B(f^{-1}(i))|}= 2^{\sum_{i \in \mathrm{im}(f)} |B(f^{-1}(i))|},
\]
where
\[
\sum_{i \in \mathrm{im}(f)} |B(f^{-1}(i))|= \sum_{i \in \mathrm{im}(f)}|\{v\in\cO: \text{$f|_{N(v)}$ is constant } i\}|= |\{v\in\cO: \text{$f|_{N(v)}$ is constant}\}|.
\]
We now count how this quantity changes between $f$ and $\widetilde{f}$. 

Assume first that $v_0 \notin K$, so $\widetilde{f} = f - 1$ on $K$ and $\widetilde{f} = f$ outside $K$. We partition $\cO$ into three sets:
\begin{itemize}
    \item $\cO\setminus N(K)$: For all $v\in \cO\setminus N(K)$, we have $N(v) \subseteq \cE\setminus K$, so $\widetilde{f}|_{N(v)} = f|_{N(v)}$. Thus, $\widetilde{f}$ is constant on $N(v)$ if and only if $f$ is.
    \item $B(K)$: For all $v\in B(K)$, we have $N(v) \subseteq K$, so $\widetilde{f}|_{N(v)} = f|_{N(v)} - 1$. Thus, $\widetilde{f}$ is constant on $N(v)$ if and only if $f$ is.
    \item $N(K) \setminus B(K)$: For all $v\in N(K) \setminus B(K)$, $N(v)$ intersects both $K$ and $\cE \setminus K$. Letting $k:=\min f|_K$, we find that $f|_{N(v)\setminus K}$ is constant $k-1$, and $f|_{N(v)\cap K}$ is constant $k$. Thus, $\widetilde{f}|_{N(v)\setminus K}$ and $\widetilde{f}|_{N(v)\cap K}$ are both equal to $k-1$. This implies that $\widetilde{f}$ is constant on $N(v)$, whereas $f$ is not.
\end{itemize}
Altogether, it follows that $|\{v\in\cO: \text{$\widetilde{f}|_{N(v)}$ is constant}\}| = |\{v\in\cO: \text{$f|_{N(v)}$ is constant}\}| + |N(K)\setminus B(K)|$, which yields $P(\widetilde{f})=2^{\,|N(K)|-|B(K)|}\,P(f)$. The other case ($v_0\in K$) follows similarly.
\end{proof}

\section{Proof of \Cref{thm:Z-hom-spectral-expanders}}\label{sec:proof-spectral}

Recall that $\mathcal{F}$ denote the family of centered legal labelings (with respect to $v_0$) of $\cE$. We further classify $\mathcal{F}$ using the notion of ``level'':

\begin{mydef}[Level with respect to a vertex]\label{def:level}
Let $\Gamma=(V, E)$, $V=\cE\cup \cO$, be a bipartite $(n,d,\lambda)$-graph. Let $\mathcal{F}$ denote the set of centered legal labelings of $\cE$ with respect to $v_0 \in \cE$. Fix $x \in \cE$. For $\ell\ge 1$, we say that $f\in \mathcal{F}$ is \emph{at level $\ell$ with respect to $x$}, written as $f\in \mathcal{F}_\ell(x)$, if there exist centered legal labelings $f_0, f_1, \dots, f_\ell$ and maximum components $K_i$ of $f_i$ for each $i=0, 1, \ldots, \ell-1$ such that
\[
f_0=f, \qquad
f_{i+1}=m_{K_i}(f_i)\ \ \text{for } 0\le i\le \ell-1,
\]
\[
x\in K_0\subseteq K_1\subseteq\cdots\subseteq K_{\ell-1},
\qquad \text{and} \qquad
|[K_{\ell-1}]|\le \frac{3n}{4}.
\]
In this case, we say that $f\in\mathcal{F}_\ell(x)$ is \textit{witnessed} by the sequence $K_0,\dots,K_{\ell-1}$.
\end{mydef}

\begin{lem}\label{lem:level-size-growth}
Let $\delta>0$, and let $\Gamma=(V, E)$ be a bipartite $(n,d,\lambda)$-graph with vertex bipartition $V=\cE\cup \cO$ and $\lambda\le (1-\delta)d$. For any $x \in \cE$, if $f\in \mathcal{F}_\ell(x)$ is witnessed by $K_0,\dots,K_{\ell-1}$, then 
\[
K_i\supseteq N^{2i}(x)
\]
for every $1\le i\leq \ell-1$. Consequently,  $|K_i|\ge \bigl(1+\frac{\delta}{4}\bigr)^{2i}$ for every $1\le i\leq \ell-1$.
\end{lem}

\begin{proof}
We show that $K_{i+1}\supseteq N^2(K_i)$ for every $0\leq i\leq \ell-1$. Since $K_i$ is a maximum component of $f_i$, the values of $f_i$ on $N^2(K_i)\setminus K_i$ are identically equal to $\min f_i|_{K_i}-1$. Thus for $f_{i+1}=m_{K_i}(f_i)$ we have $f_{i+1}|_{N^2(K_i)\setminus K_i}=\min f_{i+1}|_{K_i}$. This implies that any maximum component of $f_{i+1}$ containing $K_i$ will necessarily contain $N^2(K_i)$ as well. Since $x\in K_0$, it immediately follows by induction that $K_i\supseteq N^{2i}(x)$ for every $1\leq i\leq \ell-1$.

Finally, for all $0 \le i \le \ell-1$, $|[K_{i}]| \le |[K_{\ell-1}]|\leq \frac{3n}{4}$, so \Cref{prop:expand-3/4} guarantees that $|N^{2i}(x)| \ge (1+\frac{\delta}{4})^{2i}$, which yields $|K_i|\ge \Bigl(1+\frac{\delta}{4}\Bigr)^{2i}$ for all $1 \le i \le \ell-1$.
\end{proof}

\begin{lem}\label{lem:almost-all-good-spectral}
For every $\ell\geq 1$ and $x\in \cE$, we have
\[
\frac{\sum_{f\in\mathcal{F}_{\ell}(x)}P(f)}{\sum_{f\in\mathcal{F}}P(f)}\leq 2^{-\Omega((1+\delta/4)^{2\ell}/\log d)}.
\]
\end{lem}

\begin{proof}
Suppose $f\in\mathcal{F}_\ell(x)$ is witnessed by $K_0,\dots,K_{\ell-1}$. 

Recall from \Cref{obs:merge-preserve-labeling} that $f_0,\dots,f_\ell \in\mathcal{F}$, and the weights satisfy $P(f_i) = 2^{-|N(K_i)|+|B(K_i)|} P(f_{i+1})$.
Iterating this relation gives
\[
P(f) = P(f_\ell) \prod_{i=0}^{\ell-1} 2^{-|N(K_i)|+|B(K_i)|}.
\]

Notice that each $f \in \mathcal{F}_\ell(x)$ is determined by a final labeling $f_\ell \in \mathcal{F}$ and witnesses $K_0, \dots, K_{\ell-1}$. Therefore, we can upper bound the sum over $f \in \mathcal{F}_\ell(x)$ by summing over all possible choices of $f_\ell \in \mathcal{F}$ and nesting sequences $K_i$ (relaxing the strict condition that $K_i$ must be a maximum component). This yields
\begin{align*}
    \sum_{f\in\mathcal{F}_{\ell}(x)}P(f)
    &\leq \sum_{f_\ell\in \mathcal{F}}P(f_\ell)\sum_{\substack{x \in K_0 \sub \cdots  \sub K_{\ell-1} \\ |K_i| \ge (1+\gd/4)^{2i}\\ |[K_{i}]|\le 3n/4}}\prod_{i=0}^{\ell-1} 2^{-|N(K_i)|+|B(K_i)|} \\
    &\leq \sum_{f_\ell\in \mathcal{F}}P(f_\ell) \prod_{i=0}^{\ell-1}\left(\sum_{m\geq (1+\delta/4)^{2i}}\sum_{\substack{K_i\subseteq\cE,\,|K_i|=m\\|[K_i]|\leq 3n/4 \\ x\in K_i}}2^{-|N(K_i)|+|B(K_i)|}\right).
\end{align*}
To bound the inner sum for each $i \geq 0$, we apply \Cref{cor:H-spectral-expander}. By picking an arbitrary $v \in N(x)$, we have
\begin{align*}
    \sum_{m\geq (1+\delta/4)^{2i}}\sum_{\substack{K_i\subseteq\cE,\,|K_i|=m\\|[K_i]|\leq 3n/4 \\ x\in K_i}}2^{-|N(K_i)|+|B(K_i)|}
    &\qquad \leq \sum_{(1+\delta/4)^{2(i+1)}\leq g}\sum_{\substack{a',b,h\leq g \\ a' \le 3n/4}}\mathcal{H}(a',g,b,h,v)2^{-g+b}\\
    &\qquad \leq \sum_{(1+\delta/4)^{2(i+1)}\leq g} g^3\cdot 2^{-\Omega(t_1/\log d)}. \end{align*}
Noting that $t_1=g-a'=\Omega(g)$ (since $g \ge (1+\gd/4)a'$ so $g-a' \ge g(1-1/(1+\gd/4))=\Omega(g)$ for any fixed $\delta$), the last expression is at most
\[2^{-\Omega((1+\delta/4)^{2(i+1)}/\log d)}.\]
Substituting this bound back into our product, we obtain
\begin{align*}
    \sum_{f\in\mathcal{F}_{\ell}(x)}P(f) 
    &\leq \sum_{f_\ell\in \mathcal{F}}P(f_\ell)\prod_{i=0}^{\ell-1}2^{-\Omega((1+\delta/4)^{2(i+1)}/\log d)}\\
    &\leq \sum_{f_\ell\in \mathcal{F}}P(f_\ell) 2^{-\Omega((1+\delta/4)^{2\ell}/\log d)}.
\end{align*}
Dividing both sides by $\sum_{f\in\mathcal{F}}P(f)$ completes the proof.
\end{proof}

We are now ready to prove the main theorem.

\begin{proof}[Proof of \Cref{thm:Z-hom-spectral-expanders}] 
    Recall the definition of the range $R(h)=|\{h(v):v \in V\}|$. Since the values of $h \in \Hom_{v_0}(\Gamma)$ on adjacent vertices differ by exactly 1, the image $h(V)$ is a set of consecutive integers. Thus, $\max_{v \in V} h(v) - \min_{v \in V} h(v) = R(h) - 1$. Since every vertex in $V$ is adjacent to (or equal to) a vertex in $\cE$, we have $\max_{v \in \cE} h(v) \ge \max_{v \in V} h(v) - 1$ and $\min_{v \in \cE} h(v) \le \min_{v \in V} h(v) + 1$. Consequently, 
\[
\max_{v \in \cE} h(v) - \min_{v \in \cE} h(v) \ge R(h) - 3.
\]

Recall that the legal labeling $f$ is obtained by rescaling $h$ on $\cE$ by a factor of $1/2$. Let $m:=\max_{v \in \cE}f(v)$ and $m':=\min_{v \in \cE}f(v)$. Since $h(v) = 2f(v)$ for all $v \in \cE$, the above inequality yields $2(m - m') \ge R(h) - 3$.

Suppose $R(h) > 2t + 3$, which implies $m-m'> t$. By replacing $t$ with $3\lfloor t/3 \rfloor$ if necessary, we may assume that $t$ is a multiple of 3. By the pigeonhole principle, we either have $|f^{-1}(\{m,m-1,\dots,m-t/3\})|\leq n/2$ or $|f^{-1}(\{m',m'+1,\dots,m'+t/3\})|\leq n/2$ (since $(m - t/3) - (m' + t/3) = m - m' - 2t/3 > t/3 \ge 1$, the top and bottom intervals are disjoint). Without loss of generality, suppose the former holds.

Take any $x\in \cE$ with $f(x)=m$. For every $0\leq i\leq t/3$, let $K_i$ denote the 2-linked component of the set $\{v \in \cE : f(v) \ge m-i\}$ containing $x$. (Note that by maximality, $f < m-i$ on $N^2(K_i) \setminus K_i$.) By construction, $K_i$ is a maximum component of the appropriately merged labeling $f_i$, and we have the nested sequence $x \in K_0 \subseteq K_1 \subseteq \dots \subseteq K_{t/3-1}$. Furthermore, for all $1\leq i\leq t/3-1$, the assumption implies $|[K_i]|\leq |K_{i+1}|\leq n/2\leq 3n/4$. Thus, $f\in\mathcal{F}_{t/3}(x)$ is witnessed by $K_0,K_1,\dots,K_{t/3-1}$.

Recall from \Cref{obs:centered} that picking $h$ uniformly at random corresponds to picking $f \in \mathcal{F}$ with probability exactly proportional to its weight $P(f)$. Applying \Cref{lem:almost-all-good-spectral} (with $\ell = t/3$), we can bound the probability:
\begin{align*}
\PP_{\bf h}(R({\bf h})>2t+3) &\leq \PP_{\bf f}(m - m' \ge t) \\
&\leq \sum_{x\in \cE}\PP_{\bf f}({\bf f}\in\mathcal{F}_{t/3}(x)) \\
&\leq n\cdot 2^{-\Omega((1+\delta/4)^{2(t/3)}/\log d)} \\
&\leq n\cdot 2^{-\Omega((1+\delta/4)^{2t/3}/\log n)}.
\end{align*}
Since $d \le n$, we have $\log d \le \log n$. Setting $t=C_0 \frac{\log\log n}{\delta}$ for a sufficiently large constant $C_0=C_0(\lambda/d)>0$, the exponent dominates $\log n$, and the above probability becomes $o_n(1)$. Rescaling the constant to $C = 2C_0$, this concludes the proof that
\[
\PP(R({\bf h})\leq C\log\log n)=1-o_n(1). \qedhere
\]
\end{proof}

\section{Proof of \Cref{thm:Z-hom-middle-layers}} \label{sec:proof-middle layers}

Let $\Gamma=Q_{2d-1}^{\md}$ be the middle layers of the $(2d-1)$-dimensional Hamming cube, with vertex bipartition $V=\cE\cup \cO$ where $\cE=L_{d-1}$, $\cO=L_d$. 
Fix any $v_0\in L_{d-1}$. As in the previous section, let $\cF$ denote the family of centered legal labelings of $\cE=L_{d-1}$. 

\subsection{Setup}
We will classify legal labelings based on the existence of maximum components with ``intermediate size,'' and define a legal labeling to be \textit{max-good} (and similarly \textit{min-good}) if it lacks such components. We will show, using results from \Cref{sec:containers}, that legal labelings that are not max/min-good are amenable to a merge operation that strictly increases $P(f)$.

\begin{mydef}[Max-good]
Let $f$ be a legal labeling of $L_{d-1}$. Let $c=1/10$. We say that $f$ is \textit{max-good} if there is no maximum component $K$ such that
\beq{eq:K_size} \text{$|K|\geq d/2$ and $|[K]|\leq (1-c)\binom{2d-1}{d}$.}\enq
We denote by $\mathcal F_0 ~(\sub \cF)$ the collection of centered max-good legal labelings.
\end{mydef}

The following lemma is central to our argument, as it establishes that almost all $\mathbb Z$-homomorphisms correspond to max-good legal labelings.

\begin{lem}\label{lem:almost-all-good}
$\sum_{f\in\cF}P(f)=(1+o(1))\sum_{f\in\cF_0}P(f)$.
\end{lem}

The proof of \Cref{lem:almost-all-good} requires some preparation. As in the previous section, we classify $\cF$ using the notion of ``level," (\textit{cf.} \Cref{def:level}) but in a slightly different way.

\begin{mydef}[Level]
For every $\ell\geq 1$, we say that $f\in\cF$ is \textit{at level $\ell$} (denoted $f \in \cF_\ell$) if there is a sequence of functions $f_0,\dots,f_\ell$ such that
\begin{enumerate}[(i)]
    \item $f_0=f$, $f_\ell\in\cF_0$;
    \item for every $i=0,\dots,\ell-1$, $f_{i+1}=m_{K_i}(f_i)$ where $K_i$ is a maximum component of $f_i$ with $|K_i|\geq d/2$ and $|[K_i]|\leq (1-c)\binom{2d-1}{d}$.
\end{enumerate}
\end{mydef}

\begin{lem}\label{lem:ell_0} 
Every $f\in\cF$ belongs to some $\cF_\ell$ where $\ell$ is at most the crude bound $\ell_0:=d^2{2d-1 \choose d}$.
\end{lem}

\begin{proof} 
We may assume that $f \in \cF \setminus \cF_0$. For each $f\in\cF$, define the potential function
$$S_f := \sum_{\substack{x,y \in L_{d-1} \\ \operatorname{dist}(x,y)=2}} |f(x)-f(y)|.$$
Observe that $S_f \geq 0$. Since $\operatorname{dist}(x,y)=2$ implies $|f(x)-f(y)| \leq 1$, we have an upper bound 
$$S_f \leq |\{(x,y) \in L_{d-1} \times L_{d-1} : \operatorname{dist}(x,y)=2\}| \leq d^2 |L_{d-1}|.$$
We claim that the merge operation strictly reduces the potential $S_f$:

\begin{claim}
    For $f \in \mathcal{F} \setminus \mathcal{F}_0$, let $K$ be a maximum component of $f$ such that $|K| \geq d/2$ and $|[K]| \leq (1-c)\binom{2d-1}{d}$. Then $S_{m_K(f)} < S_f$.
\end{claim}

\begin{subproof}
We analyze the contribution of each pair $(x,y)$ with $\dist(x,y)=2$ depending on their membership in $K$:
\begin{itemize}
    \item If $x,y \in K$, then $m_K(f)(x) - m_K(f)(y) = f(x) - f(y)$.
    \item If $x,y \in L_{d-1} \setminus K$, then $m_K(f)(x) - m_K(f)(y) = f(x) - f(y)$.
    \item If $x \in K$ and $y \in L_{d-1} \setminus K$, then $y \in N^2(K) \setminus K$ (since $\dist(x,y)=2$), which implies $f(y) = f(x) - 1$. On the other hand, by the merge operation, $m_K(f)(x) = m_K(f)(y)$. Thus, $|m_K(f)(x) - m_K(f)(y)|  <  |f(x) - f(y)|$.
\end{itemize}
It follows that
$$S_{m_K(f)} = S_f - |\{(x,y) : x \in K, y \in L_{d-1} \setminus K, \operatorname{dist}(x,y)=2\}|.$$
Since $|K| < \binom{2d-1}{d}$, the set on the right-hand side is non-empty, so $S_{m_K(f)} < S_f$. 
\end{subproof}

Suppose for the sake of contradiction that a function $f \in \mathcal{F}$ is not at level $\ell$ for any $\ell \leq \ell_0$. This implies there exists a sequence of functions $f_0, \dots, f_{\ell_0+1}$ such that:
\begin{enumerate}[(i)]
    \item $f_0 = f$, and none of $f_0, \dots, f_{\ell_0}$ lie in $\mathcal{F}_0$;
    \item For every $i = 0, \dots, \ell_0$, $f_{i+1} = m_{K_i}(f_i)$, where $K_i$ is a maximum component of $f_i$ with $|K_i|\geq d/2$ and $|[K_i]|\leq (1-c)\binom{2d-1}{d}$.
\end{enumerate}
By \Cref{obs:merge-preserve-labeling}, we know that $f_i \in \mathcal{F}$ for all $i$. Furthermore, since $S_{f_0} \leq \ell_0$ and the potential strictly decreases at each step, we have
$$S_{f_{\ell_0+1}} \leq S_{f_0} - (\ell_0 + 1) \leq \ell_0 - \ell_0 - 1 = -1,$$
which contradicts the fact that $S_{f_{\ell_0+1}} \geq 0$. Thus, every $f \in \mathcal{F}$ must be at level $\ell$ for some $0 \leq \ell \leq \ell_0$.
\end{proof}

Note that a legal labeling $f$ may initially belong to multiple levels $\cF_\ell$, as there might be different sequences of mergings that turn $f$ into a max-good function. To resolve this, we uniquely assign each $f$ to $\cF_\ell$ where $\ell$ is the minimum such index. This ensures that every $f \in \cF$ has a unique membership for some $\cF_\ell$.

\subsection{Proof of the key lemma} With the setup established, we are now ready to prove \Cref{lem:almost-all-good}. 

\begin{proof}[Proof of \Cref{lem:almost-all-good}]
Consider $f\in\cF_{\ell+1}$ for some $\ell \ge 0$. Let $K$ be any maximum component of $f$ with $|K|\geq d/2$ and $|[K]|\leq (1-c)\binom{2d-1}{d}$, and again write $\widetilde f$ for $m_K(f)$ for simplicity. 
Recall from \Cref{obs:merge-preserve-labeling} that
\beq{eq:P/P} 
P(\widetilde f)/P(f)= 2^{|N(K)|-|B(K)|}.
\enq
Now, since for any $f \in \cF_{\ell+1}$, there is an $f' \in \cF_\ell$ and a 2-linked component $K\subseteq L_{d-1}$ satisfying \eqref{eq:K_size} such that $m_K(f)=f'$, we have
    \begin{align*}
        \sum_{f\in\cF_{\ell+1}}P(f)&\leq \sum_{f'\in \cF_\ell }\sum_{\substack{f\in\cF_{\ell+1}\\m_K(f)=f'\text{ for some $K$} }}P(f)\\
        &\stackrel{\eqref{eq:P/P}}{=}\sum_{f'\in \cF_\ell } P(f') \sum_{\substack{f\in\cF_{\ell+1}\\m_K(f)=f'\text{ for some $K$}  }} 2^{-|N(K)|+|B(K)|}\\
        &\leq \sum_{f'\in \cF_\ell } P(f') \sum_{\substack{K\subseteq\cE\text{ 2-linked}\\d/2\leq |K|,\,|[K]|\leq (1-c)\binom{2d-1}{d}}}2^{-|N(K)|+|B(K)|}.
    \end{align*}

Recall that $c=1/10$. Observe that as $d\to\infty$, we have $(1-c)\binom{2d-1}{d}\leq \binom{2d-1-c}{d}$; to see this, we can estimate the ratio as follows:
\begin{align*}
    \frac{\binom{2d-1-c}{d}}{\binom{2d-1}{d}}&=\frac{(2d-1-c)(2d-2-c)\cdots(d-c)}{(2d-1)(2d-2)\cdots d}\geq \left(\frac{d-c}{d}\right)^d \overset{d\to\infty}{\longrightarrow} 1/e^{c}>9/10.
\end{align*} 
Thus, for all $K$ with $|K|\ge d/2$ and $|[K]|\leq (1-c)\binom{2d-1}{d}\leq\binom{2d-1-c}{d}$, \Cref{lem:isoperimetry} implies that 
\[
|N(K)|\ge d^2/16
\qquad\text{and}\qquad
|N(K)|-|K|\ge d^2/16-d/2=\Omega(d^2).
\]
We then know from \Cref{cor:H-middle-layers} that
\begin{align*}
  \sum_{\substack{K\subseteq\cE\text{ 2-linked}\\d/2\leq |K|,\,|[K]|\leq (1-c)\binom{2d-1}{d}}}2^{-|N(K)|+|B(K)|}&\le \sum_{\substack{a',g,b,h\\g\geq d^2/100,\, a'\leq (1-c)\binom{2d-1}{d}}}|\cH(a',g,b,h)|2^{-g+b}\\
    &\leq \sum_{\substack{a',g,b,h\\g\geq d^2/100,\, a'\leq (1-c)\binom{2d-1}{d}}} 2^{-\Omega\left(\frac{t_1}{\log d}\right)} \\
    &\le \binom{2d-1}{d}^4 2^{-\Omega\left(\frac{d^2}{\log d}\right)}  =2^{-\Omega\left(\frac{d^2}{\log d}\right)}.
\end{align*}
Substituting this bound back, we obtain
\[
\frac{ \sum_{f\in\cF_{\ell+1}}P(f)}{ \sum_{f'\in\cF_{\ell}}P(f')}\leq 2^{-\Omega\left(\frac{d^2}{\log d}\right)}.
\]
Since this inequality holds for all $\ell=0,1,\dots,\ell_0$ (see \Cref{lem:ell_0} for the definition of $\ell_0$), summing over all levels yields
\[
\frac{\sum_{f\in\cF}P(f)}{\sum_{f\in\cF_{0}}P(f)}\leq \frac{\sum_{\ell=0}^{\ell_0}\sum_{f\in\cF_\ell}P(f)}{\sum_{f\in\cF_{0}}P(f)}\leq \sum_{\ell=0}^{\ell_0} \left(2^{-\Omega\left(\frac{d^2}{\log d}\right)}\right)^{\ell}\leq 1+o(1).\qedhere
\]
\end{proof}

Combining \Cref{obs:centered} and \Cref{lem:almost-all-good}, we obtain the following corollary. 

\begin{cor}\label{cor:max-good}
For almost all centered $\ZZ$-homomorphisms $h:L_{d-1}\cup L_{d}\to\ZZ$, the associated legal labeling $\frac{1}{2}h|_{L_{d-1}}$ is max-good.
\end{cor}

Similarly, we can define minimum component and min-good legal labeling.

\begin{mydef}[Minimum component]\label{def:min-component}
Let a legal labeling $f:L_{d-1}\to\ZZ$ be given. We say that $K\subseteq L_{d-1}$ is a \textit{minimum component} (of $f$) if
\begin{enumerate}[(i)]
\item $K$ is 2-linked; and
\item for every $x\in K$ and $y\in N^2(K)\setminus K$, we have $f(y)>f(x)$.
\end{enumerate}
\end{mydef}

\begin{mydef}[Min-good]
Let $f$ be a legal labeling. Let $c=1/10$. We say that $f$ is \textit{min-good} if there is no minimum component $K$ such that $|K|\geq d/2$ and $|[K]|\leq (1-c)\binom{2d-1}{d}$.
\end{mydef}

Repeating the proof of \Cref{cor:max-good} in a symmetric fashion (or, equivalently, by applying the exact same argument to the negated labeling $-f$), we arrive at the analogous statement for min-good labelings.

\begin{cor}\label{cor:min-good}
For almost all centered $\ZZ$-homomorphisms $h:L_{d-1}\cup L_{d}\to\ZZ$, the associated legal labeling $\frac{1}{2}h|_{L_{d-1}}$ is min-good.
\end{cor}

Finally, using the symmetry between $L_{d-1}$ and $L_d$ (\Cref{obs:symmetry}), we arrive at the following. (To apply the merge operation to $L_d$, we pick an arbitrary $w_0 \in N(v_0)$ to play the role of $v_0$.)
\begin{cor}\label{cor:max-min-good}
    For almost all centered $\ZZ$-homomorphisms $h:L_{d-1}\cup L_{d}\to\ZZ$, each of the associated legal labelings $\frac{1}{2}h|_{L_{d-1}}$ and $\frac{1}{2}(h|_{L_{d}}+1)$ is both max-good and min-good.
\end{cor}

\subsection{Concluding the proof of \Cref{thm:Z-hom-middle-layers}}

 We first establish the following lemma.

\begin{lem}\label{lem:three}
    Suppose $h \in \Hom_{v_0}(Q_{2d-1}^{\md})$ satisfies that
\begin{align}\label{eq:typical-max-min-good}
    \text{$\frac{1}{2}h|_{L_{d-1}}$ is a  max-good and min-good legal labeling of $L_{d-1}$.}
\end{align}
Then $h|_{L_{d}}$ takes at most three values. 
\end{lem}

To prove \Cref{lem:three}, we introduce the following definitions. Let $f = \frac{1}{2}h|_{L_{d-1}}$. For every maximum component $K \subsetneq L_{d-1}$, the \textit{enlarged maximum component} $K^{(+)}$ is the unique maximum component containing $K$ such that $\min (f|_{K^{(+)}})=\min (f|_{K})-1$; if $K=L_{d-1}$, we simply let $K^{(+)}=K$. Similarly, for every minimum component $K \subsetneq L_{d-1}$, the \textit{enlarged minimum component} $K^{(+)}$ is the unique minimum component containing $K$ such that $\max (f|_{K^{(+)}})=\max (f|_{K})+1$; again, if $K=L_{d-1}$, we let $K^{(+)}=K$.

\begin{proof}[Proof of \Cref{lem:three}]
    Let $k=\max (h|_{L_{d-1}})$ and let $K_1, K_2, \ldots$ be the maximum components in $L_{d-1}$ such that $h=k$ on each $K_i$. Since $\frac{1}{2}h|_{L_{d-1}}$ is max-good, for each $i$, either $|K_i|<d/2$ or $|[K_i]|\geq (1-c)\binom{2d-1}{d}$.

\noindent \textit{Case 1:} $|[K_{i_0}]|\geq (1-c)\binom{2d-1}{d}$ for some $i_0$.

In this case, $|N(K_{i_0})|\geq (1-c)\binom{2d-1}{d}$, so at least $(1-c)\binom{2d-1}{d}$ vertices in $L_d$ take values in $\{k+1,k-1\} \sub\{\max(h|_{L_d}),\max(h|_{L_d})-2\}$.

\noindent \textit{Case 2:} $|K_i|< d/2$ for all $i$. 

In this case, first notice that no vertices in $L_d$ take the value $k+1$, so $\max(h|_{L_d})=k-1$. Furthermore, for each $i$ we have $|K_i^{(+)}|\geq |N^2(K_i)|\geq  \Omega(d^2)$. Since $\frac{1}{2}h|_{L_{d-1}}$ is max-good, we deduce that $|[K_0^{(+)}]|\geq (1-c)\binom{2d-1}{d}$; in particular, $|N(K_0^{(+)})|\geq (1-c)\binom{2d-1}{d}$. Since $h(N(K_0^{(+)}))\subseteq \{k-1,k-3\}$, at least $(1-c)\binom{2d-1}{d}$ vertices in $L_{d}$ take values in $\{k-1,k-3\}=\{\max(h|_{L_d}),\max(h|_{L_d})-2\}$.
    
    In both cases above, at least $(1-c)\binom{2d-1}{d}$ vertices in $L_d$ take values in $\{\max(h|_{L_d}),\max(h|_{L_d})-2\}$. 
    By symmetry, using the fact that $\frac{1}{2}h|_{L_{d-1}}$ is min-good, we can also conclude that at least $(1-c)\binom{2d-1}{d}$ vertices in $L_{d}$ take values in $\{\min(h|_{L_d}),\min(h|_{L_d})+2\}$. This implies that we must have $|\max(h|_{L_d})-\min(h|_{L_d})|\le 4$ (otherwise the total number of vertices involved would exceed $|L_d|$), so $h|_{L_d}$ takes at most three values.
\end{proof}

Due to the symmetry between $L_d$ and $L_{d-1}$, we can also conclude:

\begin{lem}\label{lem:three-2}
    Suppose $h \in \Hom_{v_0}(Q_{2d-1}^{\md})$ satisfies that
\begin{align}\label{eq:typical-max-min-good-2}
    \text{$\frac{1}{2}(h|_{L_{d}}+1)$ is a  max-good and min-good legal labeling of $L_{d}$.}
\end{align}
Then $h|_{L_{d-1}}$ takes at most three values.  
\end{lem}

Combining \Cref{lem:three} and \Cref{lem:three-2}, we get that if $h$ satisfies \Cref{cor:max-min-good}, then  it takes at most six values.

Observe that $|\Hom_{v_0}(Q_{2d-1}^{\md})|\geq 2^N$ which is the number of homomorphisms taking values only in $\{-1, 0, 1\}$,  where $N := |L_{d-1}| = |L_d| = \binom{2d-1}{d-1}$.  To conclude \Cref{thm:Z-hom-middle-layers}, it remains to show that among all $h \in \Hom_{v_0}(Q_{2d-1}^{\md})$ that satisfy \Cref{cor:max-min-good} (and therefore take at most six values), only $o(|\Hom_{v_0}(Q_{2d-1}^{\md})|)$  of them take exactly six values.

\begin{lem}\label{lem:six-to-five}
    The number of $h \in \Hom_{v_0}(Q_{2d-1}^{\md})$ that satisfy \Cref{cor:max-min-good} and take exactly six values is $o(2^N)$.
\end{lem}

\begin{proof}
Without loss of generality, we may assume that $h|_{L_{d-1}}$ takes values $\{-2,0,2\}$, and $h|_{L_{d}}$ takes values $\{-3, -1, 1\}$ (precisely three values for each layer). Let $K_1, K_2, \ldots$ be the maximum components in $L_{d-1}$ with $h|_{K_i}=2$ for each $i$, and let $K'_1, K'_2, \ldots$ be the minimum components in $L_{d-1}$ with $h|_{K'_j}=-2$ for each $j$.

We consider three cases:

\noindent \textit{Case 1:} $|K'_j|<d/2$ for all $j$.

This is impossible; in this case we cannot have $h(v)=-3$ for any $v\in L_d$, meaning $h$ takes at most five values.

\noindent \textit{Case 2:} $|[K_{i_0}]|, |[K'_{j_0}]| \ge (1-c)\binom{2d-1}{d}$ for some $i_0$ and $j_0$.

This is also impossible; indeed, it would imply 
$|N(K_{i_0})\cap N(K'_{j_0})|\geq |N(K_{i_0})|+|N(K'_{j_0})|-\binom{2d-1}{d}\geq |[K_{i_0}]|+ |[K'_{j_0}]|-\binom{2d-1}{d}>0$. 
However, any vertex $v\in N(K_{i_0})\cap N(K'_{j_0})$ must satisfy both $h(v)=1$ and $h(v)\in\{-3,-1\}$, a contradiction.

\noindent \textit{Case 3:} $|K_i|<d/2$ for all $i$, and $|[K'_{j_0}]|\ge (1-c)\binom{2d-1}{d}$ for some $j_0$.

In this case, we perform a similar case analysis for $L_d$: let $R_1, R_2, \ldots$ be the minimum components in $L_d$ with $h|_{R_i}=-3$ and $R'_1, R'_2, \ldots$ be the maximum components in $L_d$ with $h|_{R'_j}=1$. By the same logic as in Cases 1 and 2, we must have
\beq{eq:Case 3} \text{$|R_i|<d/2$ for all $i$, and $|[R'_{j_1}]|\ge (1-c)\binom{2d-1}{d}$ for some $j_1$.}\enq

We give an upper bound on the number of $h$ satisfying the conditions of Case 3 and \eqref{eq:Case 3}. First observe that the 2-linked components of $h^{-1}(2)$ are precisely the $K_i$'s, and similarly the 2-linked components of $h^{-1}(-3)$ are precisely the $R_i$'s. Since these sets all have size $<d/2$, \Cref{lem:isoperimetry} implies
\beq{eq:h-1}
|h^{-1}(2)|, |h^{-1}(-3)| = O(1/d)\binom{2d-1}{d}.
\enq
On the other hand, since $N(R'_{j_1})\subseteq  h^{-1}(2)\cup h^{-1}(0)$ with $(1-c)\binom{2d-1}{d}\leq |[R'_{j_1}]|\leq |N(R'_{j_1})|$, we have
\beq{eq:h-2}
|h^{-1}(0)|\ge |N(R'_{j_1})|-|h^{-1}(2)|\geq (1-2c)\binom{2d-1}{d}.
\enq
Similarly, since $N(K'_{j_0})\subseteq  h^{-1}(-1)\cup h^{-1}(-3)$ with $(1-c)\binom{2d-1}{d}\leq |[K'_{j_0}]|\leq |N(K'_{j_0})|$,
\beq{eq:h-3}
|h^{-1}(-1)|\geq |N(K'_{j_0})|-|h^{-1}(-3)|\ge (1-2c)\binom{2d-1}{d}.
\enq
Therefore, using the entropy bound, the number of such homomorphisms $h$ is at most 
\[
\left(\sum_{i=0}^{2c N}\binom{N}{i}\right)^2\left(\sum_{i=0}^{O(N/d)}\binom{N}{i}\right)^2\leq 2^{2(H(2c)+H(O(1/d)))N}=2^{(1-\Omega(1))N} = o(2^N).  \qedhere
\]
\end{proof}

We are now ready to conclude the proof of \Cref{thm:Z-hom-middle-layers}.

\begin{proof}[Proof of \Cref{thm:Z-hom-middle-layers}]
By \Cref{cor:max-min-good}, \Cref{lem:three}, and \Cref{lem:three-2} we established that a typical $h\in \Hom_{v_0}(Q_{2d-1}^{\md})$ satisfies \Cref{cor:max-min-good} and thus takes at most six values. Furthermore, since $|\Hom_{v_0}(Q_{2d-1}^{\md})|\geq 2^N$, \Cref{lem:six-to-five} guarantees that the proportion of $h\in \Hom_{v_0}(Q_{2d-1}^{\md})$ that satisfies \Cref{cor:max-min-good} and takes exactly six values is negligible ($o_d(1)$). Therefore, with probability $1-o_d(1)$, a random $\ZZ$-homomorphism takes at most five values, yielding $R(\mathbf{h}) \le 5$. 
\end{proof}

\section*{Acknowledgement}

JP was supported by NSF Grant DMS-2324978, NSF CAREER Grant DMS-2443706 and a Sloan Fellowship.

\bibliographystyle{abbrv}

\appendix

\section{Proof of \Cref{prop:approx-quad-spectral-expander}}\label{app:1}

\begin{lem*}[\Cref{prop:approx-quad-spectral-expander}]
Let $\Gamma$ be a bipartite $d$-regular graph with vertex bipartition $V(\Gamma)=\cE\cup \cO$.
Fix $b\le h\le a'\le g$ and $v\in \cO$. Then there exists a family
\[
\mathcal U_{\mathrm{quad}}(a',g,b,h,v)\subseteq 2^{\cO}\times 2^{\cE}\times 2^{\cE}\times 2^{\cO}
\]
with
\[
|\mathcal U_{\mathrm{quad}}(a',g,b,h,v)|
\le
2^{O(g\log d/\psi)},
\]
such that every $A\in \mathcal H(a',g,b,h,v)$ has a $\psi$-approximating quadruple in $\mathcal U_{\mathrm{quad}}(a',g,b,h,v)$.
\end{lem*}

 Throughout this section, let $\Gamma$ be a $d$-regular graph with vertex bipartition $V(\Gamma)=\cE \cup \cO$. We first introduce the notion of a mutual cover.

\begin{mydef}
    For $X\subseteq \cE$ and $Y\subseteq \cO$, we say that $Y$ \textit{mutually covers} $X$ if $X\subseteq N(Y)$ and $Y\subseteq N(X)$.
\end{mydef}

We further define
\begin{align*}
    \mathcal G(a',g,v)
&:=
\bigl\{
A\subseteq \cE:\ 
A\text{ is $2$-linked},\ |[A]|=a',\ |G|=g,\ v\in G
\bigr\}.
\end{align*}

\begin{lem}\label{lem:phi-approx-spectral}
Fix $a' \le g$ and $v \in \mathcal{O}$, and let $\mathcal{G} = \mathcal{G}(a', g, v)$. Suppose $g \ge d$. There is a family $\mathcal{V} = \mathcal{V}(a', g, v) \subseteq 2^{\mathcal{O}}$ with$$|\mathcal{V}| \le 2^{ O\left( \frac{g \log^2 d}{d} \right) }$$such that for every $A \in \mathcal{G}$, $[A]$ has a mutual cover in $\mathcal{V}$.
\end{lem}

\begin{proof} We construct a small mutual cover of $[A]$ in $G$ using a randomized procedure. Include each vertex of $G$ into a random subset $Y$ independently with probability $p = \frac{\ln d}{d}$. For any $x \in [A]$, the probability that $x$ has no neighbors in $Y$ is $(1-p)^d \le e^{-pd} = \frac{1}{d}$. The expected size of $Y$ is $\frac{\ln d}{d}g$, and the expected number of vertices in $[A]$ not covered by $Y$ is at most $a'/d \le g/d$. By Markov's inequality, there exists a specific subset $Y \subseteq G$ such that:

(1) $|Y| \le \frac{2 \ln d}{d}g$, and 

(2) the set of uncovered vertices $U = [A] \setminus N(Y)$ satisfies $|U| \le \frac{2}{d}g$. 

We can complete the cover by selecting exactly one neighbor in $G$ for each $u \in U$, forming a set $Y'$. Let $V = Y \cup Y'$. Then $V \subseteq G$, $V$ is a mutual cover of $[A]$, and$$|V| \le \frac{2 \ln d + 2}{d} g = O\left( \frac{\log d}{d} g \right).$$Since $A$ is 2-linked, $[A]$ is also 2-linked, and any mutual cover of a 2-linked set is 4-linked.

Recall that the number of 4-linked subsets of $\mathcal{O}$ of size $m$ that contains a fixed vertex $v_0$ is at most $O((ed^4)^{m-1})$ (\Cref{lem:linked-set-count}).  Since there are $g$ choices for a starting vertex in $G$, the number of choices for $V$ is at most
\[
g\cdot O\left((ed^4)^{O\left( \frac{\log d}{d} g \right)}\right)=2^{O\left( \frac{\log^2 d}{d} g \right)}. \qedhere
\]
\end{proof}

We now introduce the notion of a $\psi$-approximating pair. 

\begin{mydef}\label{def:psi-approx}
    Let $A\subseteq \cE$. A \textit{$\psi$-approximating pair} of $A$ is a pair of sets $(F,S)\in 2^{\cO}\times 2^{\cE}$ such that
        \begin{enumerate}[(i)]
        \item $F\subseteq G$, $S\supseteq [A]$;
        \item for all $x\in S$, $d_F(x)\geq d-\psi$;
        \item for all $y\in \cO\setminus F$, $d_{\cE\setminus S}(y)\geq d-\psi$.
    \end{enumerate}
\end{mydef}

Note that a $\psi$-approximating quadruple of $A$ (\Cref{def:psi-approx-quad}) 
is precisely a quadruple $(F,S,P,Q)$ such that $(F,S)$ is a $\psi$-approximating 
pair of $A$, and $(P,Q)$ is a $\psi$-approximating pair of $B=B(A)$.

\begin{lem}\label{lem:psi-approx-spectral}
Fix $a' \le g$ and $v \in \mathcal{O}$, and let $\mathcal{G} = \mathcal{G}(a', g, v)$. Suppose $g \ge d$. Let $\mathcal{V}$ be as in Lemma \ref{lem:phi-approx-spectral}. For every $V \in \mathcal{V}$, there is a family $\mathcal{W} = \mathcal{W}(V) \subseteq 2^{\mathcal{O}} \times 2^{\mathcal{E}}$ with
$$|\mathcal{W}| \le 2^{O\left( \frac{g \log d}{\psi} \right) }$$
such that for every $A \in \mathcal{G}$, if $V$ is a mutual cover of $[A]$, then $A$ has a $\psi$-approximating pair in $\mathcal{W}$.
\end{lem}
\begin{proof}Let $A \in \mathcal{G}$ and suppose $V \in \mathcal{V}$ is a mutual cover of $[A]$. Let $H \subseteq [A]$ be a minimum-size set such that $F' := V \cup N(H)$ satisfies $d_{G \setminus F'}(x) \le \psi$ for all $x \in [A]$. 

We can select $H$ greedily: successively pick vertices in $[A]$ violating the condition, each adding at least $\psi$ new vertices to $F' \subseteq G$. Thus, $|H| \le g/\psi$. 
Since $H \subseteq [A]$, the number of choices for $H$ is at most $\binom{a'}{g/\psi} \le 2^{O\left( \frac{g \log d}{\psi} \right)}$.

Let $S' := \{x \in \mathcal{E} : d_{F'}(x) \ge d - \psi\}$. By definition, $[A] \subseteq S'$. Because $e(S', G) \ge (d-\psi)|S'|$ and $e(S', G) \le d|G| = dg$, we have $|S'| \le \frac{d}{d-\psi}g \le 2g$.

Next, let $U \subseteq N(S') \setminus G$ be a minimum-size set such that $S := S' \setminus N(U)$ satisfies $d_S(y) \le \psi$ for all $y \in \mathcal{O} \setminus G$. Greedily constructing $U$ yields $|U| \le \frac{|S'|}{\psi}  \le \frac{2g}{\psi}$. Since $U \subseteq N(S')$ and $|N(S')| \le d|S'| \le 2dg$, the number of choices for $U$ is $\binom{2dg}{2g/\psi} \le 2^{ O\left( \frac{g \log d}{\psi} \right)}$.

Setting $F = F' \cup \{y \in \mathcal{O} : d_S(y) > \psi\}$ gives a valid $\psi$-approximating pair $(F,S)$. The choices for $(H, U)$ uniquely determine $(F,S)$, yielding the desired bound on $|\mathcal{W}|$.\end{proof}

Combining \Cref{lem:phi-approx-spectral} and \Cref{lem:psi-approx-spectral} immediately gives the following.

\begin{cor}\label{cor:psi-approx-spectral}

Fix $a' \le g$ and $v \in \mathcal{O}$, and let $\mathcal{G} = \mathcal{G}(a', g, v)$. Suppose $g \ge d$. Then there is a family $\mathcal{U} = \mathcal{U}(a', g, v) \subseteq 2^{\mathcal{O}} \times 2^{\mathcal{E}}$ of size$$  2^{ O\left( \frac{g \log d}{\psi} \right)}$$such that every $A \in \mathcal{G}$ has a $\psi$-approximating pair $(F, S) \in \mathcal{U}$.
\end{cor}

Finally, we deduce \Cref{prop:approx-quad-spectral-expander} from \Cref{cor:psi-approx-spectral}.

\begin{proof}[Proof of \Cref{prop:approx-quad-spectral-expander}] Since $\Gamma$ is $d$-regular and bipartite, every previous definition and result for $A\subseteq \cE$ has an ``analogous'' statement for $B\subseteq\cO$. For $b\leq h$ and $w\in\cE$, let $\mathcal G(b,h,w)$ denote the collection of sets $B \subseteq \cO$ such that $B$ is $2$-linked, $|[B]|=b$, $|N(B)|=h$, and $w \in N(B)$;  let
$\mathcal U(b,h,w)\subseteq 2^{\cE}\times 2^{\cO}$ be as in the ``analogous'' statement of \Cref{cor:psi-approx-spectral}.
Then, we can construct $\cU_\qd(a',g,b,h,v)$ as follows:
\begin{itemize}
    \item Take $\mathcal U(a',g,v)$.
    \item For every $(F,S)\in \mathcal U(a',g,v)$, every $v_2\in F$, and $w\in N(v_2)$:
    \begin{itemize}
        \item Take $\mathcal U(b,h,w)$.
        \item For every $(P,Q)\in \mathcal U(b,h,w)$, add $(F,S,P,Q)$ to $\cU_\qd(a',g,b,h,v)$.
    \end{itemize}
\end{itemize}
Since $|F| \le |G| = g$ and $h \le g$, we obtain:
\[
|\cU_\qd(a',g,b,h,v)| \le 2^{O\left(\frac{g\log d}{\psi}\right)} \cdot 2^{O\left(\frac{h\log d}{\psi}\right)} \cdot gd = 2^{O\left(\frac{(g+h)\log d}{\psi}\right)} = 2^{O\left(\frac{g\log d}{\psi}\right)}.
\]

We now show that the above $\cU_\qd$ works. Fix  $A\in\cH(a',g,b,h,v)$. Since $A\in\mathcal G(a',g,v)$, there exists  $(F,S)\in \mathcal U(a',g,v)$  that is a $\psi$-approximating pair of $A$. Furthermore, since $B=[B]$, $B\subseteq F$ and $H=N(B)$, there exists $v_2\in F$ and $w\in N(v_2)$ such that $w\in H$. For that particular $w$, we have $B\in \mathcal G(b,h,w)$, so there exists   $(P,Q)\in \mathcal U(b,h,w)$ that is a $\psi$-approximating pair of $B$. This gives the desired quadruple $(F,S,P,Q)\in \cU_{\qd}$ that is a $\psi$-approximating quadruple of $A$.
\end{proof}

\section{Proof of \Cref{prop:approx-quad-middle-layers}}\label{app:2}

\begin{lem*}[\Cref{prop:approx-quad-middle-layers}]
Let $\Gamma=Q_{2d-1}^{\mathrm{mid}}$ with $\cE=L_{d-1}$ and $\cO=L_{d}$.
Fix $b\le h\le a'\le g$, and assume $g\ge d^2/100$ and $a'\leq (1-c)\binom{2d-1}{d}$.
Then there exists a family
\[
\mathcal U_{\mathrm{quad}}(a',g,b,h)\subseteq 2^{\cO}\times 2^{\cE}\times 2^{\cE}\times 2^{\cO}
\]
with
\[
|\mathcal U_{\mathrm{quad}}(a',g,b,h)|
\le
2^{O(t_1\log d/\psi)}\cdot 2^{O(t_2\log d/\psi)},
\]
such that every $A\in \mathcal H(a',g,b,h)$ has a $\psi$-approximating quadruple in $\mathcal U_{\mathrm{quad}}(a',g,b,h)$.
\end{lem*}

The main ideas  are similar to those in \cite[Section 5]{galvin2019independent}. Throughout this section, let $\Gamma=Q_{2d-1}^{\md}$ with vertex bipartition $\cE=L_{d-1}$, $\cO=L_d$.

Let $\phi:=d/2$.
We first define the notion of $\phi$-approximation.

\begin{mydef}\label{def:phi-approx}
    For $A\subseteq \cE$, set $G^\phi=G^\phi(A):=\{y\in G:d_{[A]}(y)>\phi\}$. A \textit{$\phi$-approximation} of $A$ is a set $F'\subseteq\cO$ such that $G^\phi\subseteq F'\subseteq G$ and $N(F')\supseteq [A]$.
\end{mydef}

We will obtain results analogous to \Cref{lem:phi-approx-spectral} and \Cref{lem:psi-approx-spectral}, 
with slightly different proofs.
Let $c=1/10$.

\begin{lem}\label{lem:phi-approx}
    Fix $a'\leq g$, $v\in\cO$ and let $\cG=\cG(a',g,v)$. Suppose $g\geq d^2/100$ and $a'\leq (1-c)\binom{2d-1}{d}$. Then there is a family $\cV=\cV(a',g,v)\subseteq 2^{\cO}$ with
    \[
    |\cV|\leq\begin{cases}
        2^{O(\frac{g\log^2d}{\phi d})+O(\frac{t_1\log^2d}{\phi})} & t_1=O(\frac{g(d-\phi)}{\phi d})\\
        2^{O(\frac{t_1\log^2d}{d-\phi})+O(\frac{t_1\log^2d}{\phi})} & t_1=\Omega(\frac{g(d-\phi)}{\phi d})
    \end{cases},
    \]
    such that every $A\in\cG$ has a $\phi$-approximation in $\cV$.
\end{lem}

Note that the above upper bound on $|\cV|$ is stronger than that in \Cref{lem:phi-approx-spectral} and is needed in the subsequent proof. Consequently, the proof of \Cref{lem:phi-approx} is also slightly more involved.

\begin{proof}
    Fix $A\in\cG(a',g,v)$ and set $p=\frac{20 \ln d}{\phi d}$.

\begin{claim}\label{claim:t0}
    There exists $T_0\subseteq G$ such that
    \begin{enumerate}
        \item $v\in G$,
        \item $|T_0|\leq 4gp$,
        \item $e(T_0,\cE\setminus[A])\leq 4t_1dp$,
        \item $|G^\phi\setminus N(N_{[A]}(T_0))|\leq \frac{3g}{d^{10}}$.
    \end{enumerate}
\end{claim}
\begin{subproof}
    Let $S$ be a $p$-random subset of $G$. Then $\EE(|S|)=gp$. Since $e(G,\cE\setminus[A])=gd-a'd=t_1d$, we have $\EE (e(S,\cE\setminus[A]))=t_1dp$. 

    Fix any $y\in G^\phi$, so that $\phi<|N_{[A]}(y)|\leq d$. Write $X:=N_{[A]}(y)$.
    Observe that in the middle layers of the hypercube, every two vertices $x_1, x_2\in \cE$ share at most one neighbor in $\cO$. 
    Since every vertex in $X$ is already adjacent to $y$, no two vertices in $X$ can share any other neighbor besides $y$. 
    This implies that for all $u \in N(X) \setminus \{y\}$, we must have $\deg_X(u) = 1$. 
    By summing the degrees of vertices in $X$, we get $|X|d = |X| + (|N(X)| - 1)$, which yields
    \[
        |N(X)| = 1 + |X|(d-1).
    \]
    Since $|X| > \phi = d/2$, we easily obtain
    \[
        |N(X)| > 1 + \frac{d}{2}(d-1) = \frac{d^2 - d + 2}{2} > \frac{d^2}{4} = \frac{\phi d}{2}.
    \]
    Therefore, for every $y\in G^\phi$ we have $|N(N_{[A]}(y))| > \phi d/2$.

    Because $N(N_{[A]}(y)) \subseteq G$, each vertex in this neighborhood is included in $S$ independently with probability $p$. We have
    \begin{align*}
        \EE(|G^\phi\setminus N(N_{[A]}(S))|)&=\sum_{y\in G^\phi}\PP[y\notin N(N_{[A]}(S))]
        =\sum_{y\in G^\phi}\PP[N(N_{[A]}(y))\cap S=\emptyset]\\
        &\leq g(1-p)^{\frac{\phi d}{2 }}\leq ge^{-p \frac{\phi d}{2}} = \frac{g}{d^{10}}.
    \end{align*}
    Thus, by Markov's inequality, there exists a set $T_0^*\subseteq G$ with (i) $|T_0^*|\leq 3gp$, (ii) $e(T_0^*,\cE\setminus[A])\leq 3t_1dp$, and (iii) $|G^\phi\setminus N(N_{[A]}(T_0^*))|\leq 3g/d^{10}$.

    Since $p=\Omega(\log d/d^2)$ and $g\geq d^2/100$, we have $gp\geq 1$. Also by \Cref{lem:isoperimetry}, we have $t_1=\Omega(d^2)$ which gives $t_1dp =\Omega(d^2\cdot dp) \geq d$.

    Finally, take $T_0=T_0^*\cup\{v\}$. Since $gp\geq 1$ and $t_1dp\geq d$, we have  $|T_0|\leq 4gp$ and $e(T_0,\cE\setminus[A])\leq 4t_1dp$. Thus $T_0$ satisfies (1)--(4) of the claim statement.
\end{subproof}

We now build the $\phi$-approximation using the obtained set $T_0$. Set $T_0'=G^\phi\setminus N(N_{[A]}(T_0))$  
and $L=N(N_{[A]}(T_0))\cup T_0'$. Let $T_1\subseteq G\setminus L$ be a minimum cover (see \Cref{def:bipartite-cover}) of $[A]\setminus N(L)$ in $Q_{2d-1}$. Take $F'=L\cup T_1$. 

We first verify that $F'$ is a $\phi$-approximation (\Cref{def:phi-approx}) of $A$:
\begin{itemize}
    \item Clearly $G^\phi\subseteq L\subseteq F'$. 
    \item Also $T_0'\subseteq G^\phi$, $N(N_{[A]}(T_0))\subseteq G$, and $T_1\subseteq G$, so $F'\subseteq G$.
    \item Also $N(T_1)\supseteq [A]\setminus N(L)$, so $N(F')\supseteq [A]$.
\end{itemize}

We then make the following two observations:
\begin{itemize}
    \item $F'$ is 4-linked. This is because every vertex in $F'\subseteq G$ is at distance 1 from $A$, and every vertex in $A$ is at distance 1 from $F'$ (as $N(F')\supseteq [A]$), and $A$ is 2-linked.
    \item Let $T:=T_0\cup T_0'\cup T_1$. Since $F'=N(N_{[A]}(T_0))\cup T_0'\cup T_1$, we know that every $y\in T$ is at distance $\leq 2$ from $F'$, and every $y\in F'$ is at distance $\leq 2$ from $T$. Since $F'$ is 4-linked, we get that $T$ is 8-linked.
\end{itemize}

We know from \Cref{claim:t0} that $|T_0|= O(\frac{g\log d}{\phi d})$, $|T_0'|= O(\frac{g}{d^{10}})$, and $|E(T_0,\cE\setminus[A])|= O(t_1\log d/\phi)$.
We now upper bound $|T_1|$. Note that $|E(G,\cE\setminus [A])|=gd-a'd=t_1d$ and every vertex in $G\setminus L\subseteq G\setminus G^\phi$ contributes at least $d-\phi$ edges to this set, so we have $|G\setminus L|\leq\frac{t_1d}{d-\phi}$. Since
\begin{align*}
    &d_{G\setminus L}(v)=d\quad \text{ for every } v\in [A]\setminus N(L),\\
    &d_{[A]\setminus N(L)}(u)\leq d\quad \text{ for every } u\in G\setminus L,
\end{align*}
by \Cref{lem:cover}, we have $|T_1|\leq\frac{t_1}{d-\phi}(1+\ln d)=O(\frac{t_1\log d}{d-\phi})$.

Combining the above, we know that $T=T_0\cup T_0'\cup T_1$ is an 8-linked subset of $\cO$ of size $O(\frac{g\log d}{\phi d}+\frac{g}{d^{10}}+\frac{t_1\log d}{d-\phi})$. We now begin a case discussion.

\noindent \textbf{Case 1: $t_1=O(\frac{g(d-\phi)}{\phi d})$.} In this case, we have $|T|=O(\frac{g\log d}{\phi d})$. By \Cref{lem:linked-set-count}, there are $2^{O(\frac{g\log^2 d}{\phi d})}$ choices for $T$ and thus $2^{O(\frac{g\log^2 d}{\phi d})}$ choices for the partition $(T_0,T_0',T_1)$. 
Given $T_0$, since $|E(T_0,\cE\setminus[A])|\leq O(t_1dp)=O(t_1\log d/\phi)$ and $|N(T_0)| \le d|T_0| = O(\frac{g\log d}{\phi})$, the number of choices for $N_{\cE\setminus[A]}(T_0)$ is at most
\[
\sum_{i=0}^{t_1\log d/\phi}\binom{N(T_0)}{i}=2^{O\left(\frac{t_1\log^2 d}{\phi}\right)},
\]
where the equality uses the fact that $t_1 = \Omega(g/d)$ (by \Cref{lem:isoperimetry}), which implies $\log(|N(T_0)|/i) = O(\log d)$. Since specifying $N_{\cE\setminus[A]}(T_0)$ uniquely determines $N_{[A]}(T_0) = N(T_0)\setminus N_{\cE\setminus[A]}(T_0)$, there is the same number of choices for $N_{[A]}(T_0)$. Since $F'$ is completely determined by $T_0',T_1,$ and $N_{[A]}(T_0)$, we get that there are 
\[
2^{O\left(\frac{g\log^2 d}{\phi d}\right)}\cdot 2^{O\left(\frac{t_1\log^2 d}{\phi}\right)}
\]
choices for $F'$.

\noindent \textbf{Case 2: $t_1=\Omega(\frac{g(d-\phi)}{\phi d})$.} In this case, we have $|T|=O(\frac{t_1\log d}{d-\phi})$. By \Cref{lem:linked-set-count}, there are $2^{O(\frac{t_1\log^2 d}{d-\phi})}$ choices for $T$ and thus $2^{O(\frac{t_1\log^2 d}{d-\phi})}$ choices for the partition $(T_0,T_0',T_1)$. 
Following the exact same reasoning as in Case 1, given $T_0$, there are $2^{O(\frac{t_1\log^2 d}{\phi})}$ choices for $N_{[A]}(T_0)$. Since $F'$ is completely determined by $T_0',T_1,$ and $N_{[A]}(T_0)$, we get that there are 
\[
2^{O\left(\frac{t_1\log^2 d}{d-\phi}\right)}\cdot 2^{O\left(\frac{t_1\log^2 d}{\phi}\right)}
\]
choices for $F'$.
\end{proof}

\begin{lem}\label{lem:psi-approx}
Fix $a'\leq g$, $v\in\cO$ and let $\cG=\cG(a',g,v)$.
Suppose $g\geq d^2/100$ and $a'\leq (1-c)\binom{2d-1}{d}$.
Let $\cV$ be as in \Cref{lem:phi-approx}. For every $F'\in\mathcal V$, there is a family $\mathcal W=\mathcal W(F')\subseteq 2^{\cO}\times 2^\cE$ with
    \[
    |\mathcal W|\leq 2^{O\left(\frac{t_1d\log d}{(d-\phi)\psi}\right)+O\left(\frac{t_1d\log d}{(d-\psi)\psi}\right)},
    \]
    such that for every $A\in\mathcal G$, if $F'$ is a $\phi$-approximation of $A$, then $A$ has a $\psi$-approximating pair in $\mathcal W$. Furthermore, for every $(F,S)\in\cW$, we have $|S|\leq d^3g$.
\end{lem}
\begin{proof}
Let $\cV=\cV(a',g,v)$ be as in \Cref{lem:phi-approx}. Fix $F'\in\cV$. Consider any $A\in\cG$ such that $F'$ is a $\phi$-approximation of $A$. We produce a $\psi$-approximation $(F,S)$ of $A$ by running the following algorithm.

\textbf{Step 1.} If there is $u\in[A]$ such that $d_{G\setminus F'}(u)>\psi$, pick the (lexicographically) smallest $u$ and update $F'$ by $F'=F'\cup N(u)$. Repeat until there is no such $u$. Then set $F''=F'$ and $S''=\{u\in \cE:d_{F''}(u)\geq d-\psi\}$. Go to Step 2. 

\textbf{Step 2.} If there is $w\in \cO\setminus G$ such that $d_{S''}(w)>\psi$, pick the smallest $w$ and update $S''$ by $S''=S''\setminus N(w)$. Repeat until there is no such $w$. Then set $S=S''$ and $F=F''\cup \{w\in \cO: d_{S}(w)>\psi\}$. Output $(F,S)$.

\begin{claim}\label{claim:psi-1}
    The obtained pair $(F,S)$ is a $\psi$-approximation of $A$ with $|S|\leq d^3g$.
\end{claim}
\begin{subproof}
    We check the following:
    \begin{itemize}
        \item $F\subseteq G$ and $S\supseteq[A]$. 
        
        At the end of Step 1, we have $d_{G\setminus F'}(u)\leq \psi$ for all $u\in[A]$, so $[A]\subseteq S''$. Since Step 2 only deletes vertices in $N(\cO\setminus G)$ (which are outside $[A]$) from $S''$, at the end of Step 2 we have $[A]\subseteq S$.

        It is clear that $F''\subseteq G$ at the end of Step 1. At the end of Step 2, we have $d_{S''}(w)\leq\psi$ for all $w\in\cO\setminus G$, so $\{w\in \cO: d_{S}(w)>\psi\}\subseteq G$ as well. Thus $F\subseteq G$. 
        
        \item $d_F(u)\geq d-\psi$ for all $u\in S$. This is because at the end of Step 1 we have $d_{F''}(u)\geq d-\psi$ for all $u\in S''$. At the end of Step 2, we have $S\subseteq S''$ and $F\supseteq F''$.
        
        \item $d_{\cE\setminus S}(w)\geq d-\psi$ for all $w\in \cO\setminus F$. At the end of Step 2, for all $w\in \cO\setminus F$ we have $d_{S}(w)\leq \psi$, which gives $d_{\cE\setminus S}(w)\geq d-\psi$. \qedhere
    \end{itemize}
\end{subproof}

\begin{claim}\label{claim:psi-iterations-bound}
    The above Step 1 has at most $\frac{t_1d}{(d-\phi)\psi}$ iterations, and Step 2 has at most $\frac{t_1d}{(d-\psi)\psi}$ iterations.
\end{claim}
\begin{subproof}
    Initially, we have $|G\setminus F'|\leq \frac{t_1d}{d-\phi}$; this is because every $v\in G\setminus F'\subseteq G\setminus G^\phi$ contributes at least $d-\phi$ edges to $E(\cE\setminus [A], G)$, which is a set of edges of size $dg-da'=t_1d$. Since each iteration in Step 1 reduces $|G\setminus F'|$ by at least $\psi$, there are at most $\frac{t_1d}{(d-\phi)\psi}$ iterations.

    At the beginning of Step 2, we have $|S''\setminus [A]|\leq\frac{t_1d}{d-\psi}$; this is because every $u\in S''\setminus [A]$ contributes at least $d-\psi$ edges to $E(\cE\setminus [A], G)$, a set of size $t_1d$. Since each iteration in Step 2 reduces $|S''\setminus [A]|$ by at least $\psi$, there are at most $\frac{t_1d}{(d-\psi)\psi}$ iterations.
\end{subproof}

\begin{claim}\label{claim:psi-2}
    The above algorithm has at most 
    \[
    2^{O\left(\frac{t_1d\log d}{(d-\phi)\psi}\right)+O\left(\frac{t_1d\log d}{(d-\psi)\psi}\right)}
    \]
    outputs as $A$ ranges over all sets in $\cG(a',g,v)$ that have $F'$ as a $\phi$-approximation.
\end{claim}
\begin{subproof}
Since $F'$ is fixed, the output of Step 1 is completely determined by the ordered sequence of $u$'s whose neighborhoods $N(u)$ are added to $F'$. Given the output of Step 1, the output of Step 2 is completely determined by the ordered sequence of $w$'s whose neighborhoods $N(w)$ are removed from $S''$.

By \Cref{claim:psi-iterations-bound}, the set of $u$'s is a subset of $[A]\subseteq N(F')$ of size at most $\frac{t_1d}{(d-\phi)\psi}$. Since $|N(F')|\leq|N(G)|\leq dg$, the number of possible outputs of Step 1 is at most
\[
\sum_{i=0}^{\frac{t_1d}{(d-\phi)\psi}}\binom{dg}{i}=2^{O\left(\frac{t_1d\log d}{(d-\phi)\psi}\right)},
\]
where the equality follows from $g/t_1 = O(d)$ (by \Cref{lem:isoperimetry}), which implies $\log(dg/i) = O(\log d)$. 

Fix any output $(F'',S'')$ of Step 1. Then the set of $w$'s is a subset of $N(S'')$ of size at most $\frac{t_1d}{(d-\psi)\psi}$. Since $S''\subseteq N(F'')$ where $F''\subseteq F'\cup N([A])\subseteq F'\cup N^2(F')$, we have $|S''|\leq |N^3(F')|\leq d^3g$. Thus, the number of possible outputs of Step 2 is at most
\[
\sum_{i=0}^{\frac{t_1d}{(d-\psi)\psi}}\binom{d^3g}{i}=2^{O\left(\frac{t_1d\log d}{(d-\psi)\psi}\right)},
\]
again using the fact that $\log(d^3g/i) = O(\log d)$. Furthermore, the final $S\subseteq S''$ satisfies $|S|\leq |S''|\leq d^3g$.
\end{subproof}

Thus, for every $F'\in\cV$, we can take $\cW(F')$ to be the set of possible outputs $(F,S)$ in the above algorithm.
\end{proof}

Combining \Cref{lem:phi-approx} and \Cref{lem:psi-approx} immediately gives the following. (Recall from \Cref{lem:isoperimetry} that $a'\leq (1-c)\binom{2d-1}{d}$ implies $t_1=g-a'=\Omega(g/d)$.) 

\begin{cor}\label{cor:psi-approx}
Fix $a'\leq g$, $v\in\cO$ and let $\cG=\cG(a',g,v)$. Suppose $g\geq d^2/100$ and $a'\leq (1-c)\binom{2d-1}{d}$.
    Then there is a family $\mathcal U=\cU(a',g,v)\subseteq 2^{\cO}\times 2^{\cE}$ of size
    \[
    2^{O\left(\frac{t_1\log d}{\psi}\right)}
    \]
    such that every $A\in\cG$ has a $\psi$-approximating pair $(F,S)$ in $\mathcal U$. Furthermore, for every $(F,S)\in\cU$, we have $|S|\leq d^3g$.
\end{cor}

An almost direct implication of \Cref{cor:psi-approx} is the following.
\begin{cor}\label{cor:psi-approx-quad}
Fix $b\leq h\leq a'\leq g$. Suppose $g,h\geq d^2/100$ and $a'\leq (1-c)\binom{2d-1}{d}$.
Then there is a family $\mathcal U_{\qd}=\cU_\qd(a',g,b,h)\subseteq 2^{\cO}\times 2^{\cE}\times 2^{\cE}\times 2^{\cO}$ of size
    \[
    2^{O\left(\frac{t_1\log d}{\psi}\right)}\cdot 2^{O\left(\frac{t_2\log d}{\psi}\right)}
    \]
    such that every $A\in\cH(a',g,b,h)$ has a $\psi$-approximating quadruple in $\mathcal U_{\qd}$.
\end{cor}
\begin{proof}
Note that $g\geq h\geq d^2/100$ and $b\leq a'\leq (1-c)\binom{2d-1}{d}$. 

Due to the symmetry between $\cE$ and $\cO$ (\Cref{obs:symmetry}), every previous definition and result for $A\subseteq \cE=L_{d-1}$ has an ``analogous'' statement for $B\subseteq \cO=L_d$. For every $a'\leq g$ with $g\geq d^2/100$, $a'\leq (1-c)\binom{2d-1}{d}$ and $v\in \cO$, let $\cU(a',g,v)\subseteq 2^{\cO}\times 2^{\cE}$ be as in \Cref{cor:psi-approx}; for every $b\leq h$ with $h\geq d^2/100$, $b\leq (1-c)\binom{2d-1}{d}$ and $v'\in\cE$, let $\cU(b,h,v')\subseteq 2^{\cE}\times 2^{\cO}$ be as in the ``analogous'' statement of \Cref{cor:psi-approx}.
Let
\[
\cU_{\qd}=\cU_{\qd}(a',g,b,h):=\bigcup_{v\in \cO,\, v'\in\cE}\{(F,S,P,Q):(F,S)\in \cU(a',g,v),\, (P,Q)\in \cU(b,h,v')\}.
\]

We verify that the above $\cU_{\qd}$ satisfies the desired properties. Clearly we have 
\[
|\cU_{\qd}|\leq |\cO||\cE||\cU(a',g,v)||\cU(b,h,v')|\leq 2^{O(d)}\cdot 2^{O\left(\frac{t_1\log d}{\psi}\right)}\cdot 2^{O\left(\frac{t_2\log d}{\psi}\right)} = 2^{O\left(\frac{t_1\log d}{\psi}\right)}\cdot 2^{O\left(\frac{t_2\log d}{\psi}\right)}.
\]
Here in the last equality, we crucially use the assumption that \textit{both} $g,h\geq d^2/100$, which gives $t_1,t_2=\Omega(d^2)$ according to \Cref{lem:isoperimetry}(b) (thus absorbing the $2^{O(d)}$ factor).

Furthermore, for every $A\in\cH(a',g,b,h)$, there exist $v\in\cO$ and $v'\in\cE$ such that $A\in\cG(a',g,v)$ and $B\in\cG(b,h,v')$. By \Cref{cor:psi-approx}, there exists $(F,S)\in\cU(a',g,v)$ that is a $\psi$-approximating pair of $A$ (i.e., meets (i)--(iii) in \Cref{def:psi-approx-quad}), and there exists $(P,Q)\in\cU(b,h,v')$ that is a $\psi$-approximating pair of $B$ (i.e., meets (iv)--(vi) in \Cref{def:psi-approx-quad}). Thus the tuple $(F,S,P,Q)$ is a $\psi$-approximating quadruple of $A$ that lies in the constructed $\cU_{\qd}$.
\end{proof}

We further prove a statement similar to \Cref{cor:psi-approx-quad}, but valid for the complement range of $h$.

\begin{cor}\label{cor:psi-approx-quad-2}
    Fix $b\leq h\leq a'\leq g$. Suppose $g\geq d^2/100$, $h\leq d^2/100$ and $a'\leq (1-c)\binom{2d-1}{d}$. Then there is a family $\mathcal U_{\qd}=\cU_{\qd}(a',g,b,h)\subseteq 2^{\cO}\times 2^{\cE}\times 2^{\cE}\times 2^{\cO}$ of size
    \beq{eq:h_small}
    2^{O\left(\frac{t_1\log d}{\psi}\right)}
    \enq
    such that every $A\in\cH(a',g,b,h)$ has a $\psi$-approximating quadruple in $\mathcal U_{\qd}$.
\end{cor}
\begin{proof} 
Given $v\in \cO$, consider $\cU(a',g,v)$ as in \Cref{cor:psi-approx}. 
Fix $A\in \bigcup_{v\in \cO}\cG(a',g,v)$. Consider any $\psi$-approximating pair $(F,S)$ of $A$ in $\cU(a',g,v)$, recalling that $|S|\leq d^3g$ (by the last part of \Cref{lem:psi-approx}). Observe that if we simply take $P=H$ and $Q=B$, then
$(F,S,H,B)$ is already a $\psi$-approximating quadruple of $A$. Since $H\subseteq A\subseteq S$ with $|H|= h$ and $|S|\leq d^3g$, for every fixed $S$, the number of choices for $H$ is at most
    \[
    \binom{d^3g}{h}= 2^{O(h\log g)}.
    \]
    After $H$ is fixed, we can obtain the set $Q':=\{y\in \cO : N(y)\subseteq H\}$ with $|Q'|\leq |H|=h$. Now we can fix $B\subseteq Q'$, and the number of choices for $B$ is at most $2^h$.
    
    Thus, for every $(F,S)\in\bigcup_{v\in \cO}\cU(a',g,v)$, the number of choices for $(F,S,H,B)$ is at most 
    \[
    2^{O(h\log g)}\cdot 2^h=2^{O(h\log g)}.
    \]
    This gives a family $\mathcal U_{\qd}=\cU_{\qd}(a',g,b,h)\subseteq 2^{\cO}\times 2^{\cE}\times 2^{\cE}\times 2^{\cO}$ with 
    \[
    |\mathcal U_{\qd}|\leq \left|\bigcup_{v\in \cO}\cU(a',g,v)\right|\cdot 2^{O(h\log g)}= 2^{O(d)}\cdot 2^{O\left(\frac{t_1\log d}{\psi}\right)}\cdot 2^{O(h\log g)}
    \]
    such that every $A\in\cH(a',g,b,h)$ has a $\psi$-approximating quadruple in $\mathcal U_{\qd}$. 
    
    Finally, to derive \eqref{eq:h_small}, we consider two cases.
    
    First, if $d^2/100\leq g\leq d^6$, then we have (from \Cref{lem:isoperimetry}(b)) that $t_1=\Omega(g)=\Omega(d^2)$ and $h\leq a'=O(g/d)$. This tells us that  
    \[
    |\cU_{\qd}|= 2^{O(d)}\cdot 2^{O\left(\frac{t_1\log d}{\psi}\right)}\cdot 2^{O(h\log d)}=2^{O\left(\frac{t_1\log d}{\psi}\right)}.
    \]

    On the other hand, if $g> d^6$, then we know from \Cref{lem:isoperimetry}(a) that $t_1=\Omega(g/d)=\Omega(d^5)$. Since $h\leq d^2/100= O(t_1/d^2)$, we get that
    \[
    |\cU_{\qd}|= 2^{O(d)}\cdot 2^{O\left(\frac{t_1\log d}{\psi}\right)}\cdot 2^{O(hd)}=2^{O\left(\frac{t_1\log d}{\psi}\right)}. \qedhere
    \]
\end{proof}

Combining \Cref{cor:psi-approx-quad} and \Cref{cor:psi-approx-quad-2} gives \Cref{prop:approx-quad-middle-layers}.

\section{Proof of \Cref{lem:reconstruction}}\label{app:reconstruction}\label{app:3}

\begin{prop*}[\Cref{lem:reconstruction}]
Let $\Gamma$ be a bipartite $d$-regular graph with vertex bipartition $V(\Gamma)=\cE\cup \cO$.
Assume $\psi=\sqrt d$, and assume in addition that
\[
t_1=\Omega(g/d),
\qquad
t_2=\Omega(h/d).
\]
Fix $b\le h\le a'\le g$. Then for every quadruple $(F,S,P,Q)$ satisfying conditions (2), (3), (5), and (6) of \Cref{def:psi-approx-quad}, the number of $2$-linked sets $A\subseteq \cE$ with
\[
|[A]|=a',
\qquad
|G|=g,
\qquad
|B|=b,
\qquad
|H|=h,
\]
for which $(F,S,P,Q)$ also satisfies conditions (1) and (4) of \Cref{def:psi-approx-quad} is at most
\[
2^{\,g-b- \Omega(t_1/\log d)- \Omega(t_2/\log d)}.
\]
\end{prop*}

We start with the following bound on the sizes of sets in a $\psi$-approximating quadruple.

\begin{lem}\label{lem:S-bound}
    If $(F,S,P,Q)$ is a $\psi$-approximating quadruple of $A$, then we have
    \begin{align*}
        &|S|\leq |F|+O(\psi t_1 /(d-\psi)),\\
        &|Q|\leq |P|+O(\psi t_2 /(d-\psi)).
    \end{align*}
    In particular we have $|S|=O(g)$ and $|Q|=O(h)$.
\end{lem}
\begin{proof}
Observe that we have $e(S,G)\leq d|F|+\psi|G\setminus F|$ and $e(S,G)\geq d|[A]|+(d-\psi)|S\setminus [A]|=d|S|-\psi|S \setminus [A]|$. These two give
\[
|S|\leq |F|+\frac{\psi}{d}(|G\setminus F|+|S\setminus [A]|).
\]
Since each vertex of $(G\setminus F)\cup (S\setminus [A])$ contributes at least $d-\psi$ edges to $e(G, \cE\setminus [A])$, a set of size $gd-a'd=t_1d$, we get the first bound. The second inequality follows analogously.
\end{proof}

\begin{proof}[Proof of \Cref{lem:reconstruction}]
    Recall from the assumption that $t_1=g-a'=\Omega(g/d)$ and $t_2=h-b=\Omega(h/d)$. Given $(F,S,P,Q)$ and $b\leq h\leq a'\leq g$, we say that $Q$ is tight if $|Q|<b+\frac{t_2}{10\log d}$ and slack otherwise; we say that $S$ is tight if $|S|<g- \frac{t_1}{10\log d}$ and slack otherwise.

    We now define the following procedure to specify $A$. 
    \begin{itemize}
        \item We first pick some $D\subseteq A$ as follows. If $Q$ is tight, we pick $B\subseteq Q$ with $|B|=b$ and take $D= N(B)\subseteq A$; if $Q$ is slack we take $D=P$.
        \item If $S$ is tight, we choose $A\setminus D \subseteq S\setminus D$.
        \item If $S$ is slack, we first pick $G$ by picking $G\setminus F\subseteq N(S)\setminus F$. In this case $|S|\geq g-\frac{t_1}{10\log d}$. Then we know from \Cref{lem:S-bound} that $|F|\geq |S|-O(\psi t_1/d)\geq g-\frac{t_1}{5\log d}$. Thus $|G\setminus F|\leq \frac{t_1}{5\log d}$.
        
        Observe that $G$ determines $[A]$. We then choose $A\setminus D\subseteq [A]\setminus D$.
    \end{itemize}

    We have the following observations.
    \begin{itemize}
        \item If $Q$ is tight, then the number of choices of $D$ is at most 
        \[
        \binom{O(h)}{\leq \frac{t_2}{10\log d}}\leq 2^{\frac{t_2}{10\log d}\cdot \log(10h\log d/t_2)}.
        \]
        In this case $|D|=h$.
        \item If $Q$ is slack there is 1 choice for $D$. In this case $|D|=|P|\geq |Q|-O(\psi t_2/(d-\psi))\geq b+ \frac{t_2}{10\log d}-O(\psi t_2/(d-\psi))\geq b+ \frac{t_2}{20\log d}$.
        \item If $S$ is slack, then since $|N(S)|\leq d|S|=O(gd)$, the number of possible choices for $G\setminus F$ is at most
        \[
        \binom{O(gd)}{\leq \frac{t_1}{5\log d}}\leq 2^{ \frac{t_1}{5\log d}\cdot \log(5gd\log d/t_1)}.
        \]
    \end{itemize}

    Finally we bound the number of choices in each of the four cases.
    \begin{enumerate}
        \item Suppose both $S$ and $Q$ are tight. In this case $|S\setminus D|\leq g- \frac{t_1}{10\log d}-h$. Thus the number of choices of $A$ is at most
        \[
        2^{\frac{t_2}{10\log d}\cdot \log(10h\log d/t_2)}\cdot 2^{g- \frac{t_1}{10\log d}-h}\leq 2^{\frac{t_2}{10\log d}\cdot \log(10h\log d/t_2)}\cdot 2^{g- \frac{t_1}{10\log d}-b-t_2}\leq 2^{g-b-\frac{t_1}{10\log d}-t_2/2}.
        \]
        \item Suppose $S$ is tight and $Q$ is slack. In this case $|S\setminus D|\leq g-\frac{t_1}{10\log d}-b-\frac{t_2}{20\log d}$. So the number of choices of $A$ is at most $2^{g-\frac{t_1}{10\log d}-b-\frac{t_2}{20\log d}}$.
        \item Suppose $Q$ is tight and $S$ is slack. In this case $|[A]\setminus D|=a'-h$. Then the number of choices of $A$ is 
        \begin{align*}
            &2^{\frac{t_2}{10\log d}\cdot \log(10h\log d/t_2)}\cdot 2^{\frac{t_1}{5\log d}\cdot \log(5gd\log d/t_1)}\cdot 2^{a'-h}\\
            &< 2^{\frac{t_2}{10\log d}\cdot \log(10h\log d/t_2)}\cdot 2^{\frac{t_1}{5\log d}\cdot \log(5gd\log d/t_1)}\cdot 2^{g-t_1-b-t_2}\\
            &\leq 2^{g-b-t_1/2-t_2/2}.
        \end{align*}
        \item Suppose both $Q$ and $S$ are slack. In this case $|[A]\setminus D|\leq a'-b-\frac{t_2}{20\log d}$. Then the number of choices of $A$ is 
       \begin{align*}
           & 2^{ \frac{t_1}{5\log d}\cdot \log(5gd\log d/t_1)}\cdot 2^{a'-b-\frac{t_2}{20\log d}} = 2^{ \frac{t_1}{5\log d}\cdot \log(5gd\log d/t_1)}\cdot 2^{g-t_1-b-\frac{t_2}{20\log d}}\\
           &\leq 2^{g-b-t_1/2-\frac{t_2}{20\log d}}.
       \end{align*}
    \end{enumerate}
    Thus, we conclude that the number of choices of $A$ is always $2^{g-b-\Omega(t_1/\log d)-\Omega(t_2/\log d)}$.
\end{proof}

\end{document}